\documentclass[11pt]{article}
\usepackage{amssymb,amsmath,amsthm}
\usepackage{chemarr}
\usepackage{color}
\usepackage{marginnote}
\usepackage{pict2e,color}

\numberwithin{equation}{section}
\newtheorem{thm}{Theorem}[section]

\newtheorem{prop}[thm]{Proposition}
\newtheorem{lem}[thm]{Lemma}
\newtheorem{cor}[thm]{Corollary}
\theoremstyle{definition}
\newtheorem{rem}[thm]{Remark}

\let\oldproofname=\proofname
\renewcommand{\proofname}{\rm\bf{\oldproofname}}

\newcommand{\N}{\mathbb{N}}

\newcommand{\R}{\mathbb{R}}
\newcommand{\C}{\mathbb{C}}
\newcommand{\T}{\mathbb{T}}

\newcommand{\cA}{\mathcal{A}}
\newcommand{\cB}{\mathcal{B}}

\newcommand{\cE}{\mathcal{E}}

\newcommand{\cM}{\mathcal{M}}
\newcommand{\cN}{\mathcal{N}}
\newcommand{\cO}{\mathcal{O}}

\newcommand{\cS}{\mathcal{S}}

\newcommand{\dd}{\,{\rm d}}
\newcommand{\D}{{\rm d}}
\renewcommand{\div}{\mathop{\mathrm{div}}\nolimits}

\newcommand{\bu}{\mathrm{bu}}

\newcommand{\QED}{\mbox{}\hfill$\Box$}
\renewcommand{\:}{\thinspace :}

\newcommand{\cbf}{\mathbf{c}}
\newcommand{\fbf}{\mathbf{f}}
\newcommand{\ul}{\mathrm{ul}}
\newcommand{\1}{\mathbf{1}}

\begin{document}

\title{Diffusive relaxation to equilibria for an extended \\ 
reaction-diffusion system on the real line}

\author{{\bf Thierry Gallay} and {\bf Sini\v{s}a Slijep\v{c}evi\'c}}

\date{June 29, 2021}

\maketitle

\begin{abstract}
We study the long-time behavior of the solutions of a two-component
reaction-diffusion system on the real line, which describes the basic chemical
reaction $\cA \xrightleftharpoons[]{} 2 \cB$.  Assuming that the initial
densities of the species $\cA, \cB$ are bounded and nonnegative, we prove that
the solution converges uniformly on compact sets to the manifold $\cE$ of all
spatially homogeneous chemical equilibria. The result holds even if the species
diffuse at very different rates, but the proof is substantially simpler for
equal diffusivities. In the spirit of our previous work on extended dissipative 
systems \cite{GS1}, our approach relies on localized energy estimates, and
provides an explicit bound for the time needed to reach a neighborhood of the
manifold $\cE$ starting from arbitrary initial data. The solutions 
we consider typically do not converge to a single equilibrium as 
$t \to +\infty$, but they are always quasiconvergent in the sense that 
their $\omega$-limit sets consist of chemical equilibria.
\end{abstract}

\section{Introduction}\label{sec1}

Reaction-diffusion systems satisfying a detailed or complex balance condition
provide interesting examples of evolution equations where the qualitative
behavior of the solutions can be studied using entropy methods. Such systems
typically describe reversible chemical reactions of the form
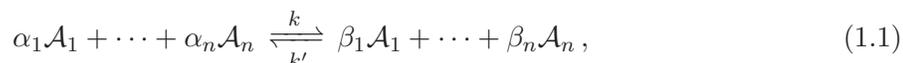
\begin{equation}\label{chemreact}
  \alpha_1 \cA_1 + \dots + \alpha_n \cA_n \,\xrightleftharpoons[~k'~]{k}\,
  \beta_1 \cA_1 + \dots + \beta_n \cA_n\,, 
\end{equation}
where $\cA_1, \dots, \cA_n$ denote the reactant and product species, $k, k' > 0$
are the reaction rates, and the nonnegative integers $\alpha_i, \beta_i$
($i = 1,\dots,n$) are the stoichiometric coefficients. According to the 
law of mass action, the concentration $c_i(x,t)$ of the species $\cA_i$ 
satisfies the reaction-diffusion equation
\begin{equation}\label{RDgen}
  \partial_t c_i \,=\, d_i \Delta c_i + (\beta_i - \alpha_i) 
  \biggl(k\prod_{j=1}^n c_j^{\alpha_j} - k' \prod_{j=1}^n c_j^{\beta_j}
  \biggr)\,, \qquad i = 1,\dots,n\,,
\end{equation}
where $\Delta$ is the Laplace operator acting on the space variable $x$, and
$d_i > 0$ denotes the diffusion coefficient of species $\cA_i$.  We refer the
reader to \cite{HJ,SRJ,DFT,Mi2} for a more detailed mathematical modeling of
chemical reactions, including the realistic situation where several reactions
occur at the same time. For general kinetic systems, there is a notion of {\em
  detailed balance}, which asserts that all reactions are reversible and
individually in balance at each equilibrium state, and a weaker notion of {\em
  complex balance}, which only requires that each reactant or product complex is
globally at equilibrium if all reactions are taken into account.  In the present
paper, we focus on a particular example of the single-reaction system
\eqref{RDgen}, for which the detailed balance condition is automatically
satisfied.

In recent years, many authors investigated the long-time behavior of solutions
to reaction-diffusion systems with complex or detailed balance, assuming that
the reaction takes place in a bounded domain $\Omega \subset \R^N$ and using an
entropy method that we briefly explain in the case of system \eqref{RDgen} with
$k = k'$. If $\cbf(t) = (c_1(t),\dots,c_n(t))$ is a solution of \eqref{RDgen} in
$\Omega$ satisfying no-flux boundary conditions on $\partial\Omega$, we have the
entropy dissipation law $\frac{\D}{\D t}\Phi(\cbf(t)) = -D(\cbf(t))$, where $\Phi$ 
is the entropy function defined by
\begin{equation}\label{Phidef}
  \Phi(\cbf) \,=\, \sum_{i=1}^n \int_\Omega \phi\bigl(c_i(x)\bigr)\dd x\,, \qquad
  \phi(z) = z\log(z)-z+1\,,
\end{equation}
and $D$ is the entropy dissipation
\begin{equation}\label{Ddef}
  D(\cbf) \,=\, \sum_{i=1}^n d_i \int_\Omega \frac{|\nabla c_i(x)|^2}{c_i(x)}
  \dd x + k \int_\Omega \log\biggl(\frac{B(x)}{A(x)}\biggr)\Bigl(B(x) - A(x)
  \Bigr)\dd x\,,
\end{equation}
where $A(x) = \prod c_j(x)^{\alpha_j}$, $B(x) = \prod c_j(x)^{\beta_j}$. It is
clear from \eqref{Ddef} that the entropy dissipation $D(\cbf)$ is nonnegative
and vanishes if and only if the concentrations $c_i$ are spatially homogeneous
($\nabla c_i = 0$) and the system is at chemical equilibrium ($A = B$). The
entropy is therefore a Lyapunov function for \eqref{RDgen}, and using LaSalle's
invariance principle one deduces that all bounded solutions converge to
homogeneous chemical equilibria as $t \to +\infty$ \cite{Gr1,Ro}.  In addition,
under appropriate assumptions, the entropy dissipation $D(\cbf)$ can be bounded
from below by a multiple of the entropy $\Phi(\cbf)$, or more precisely of the
{\em relative entropy} $\Phi(\cbf\,|\,\cbf_*)$ with respect to some equilibrium
$\cbf_*$. Such a lower bound can be established using a compactness argument
\cite{GGH,Gr2}, or invoking functional inequalities such a the logarithmic
Sobolev inequality \cite{DF1,DF2,DFT,FT1,FT2,Mi2,PSZ}. This leads to a first order
differential inequality for the relative entropy, which implies exponential
convergence in time to equilibria. In its constructive form, this
entropy-dissipation approach even provides explicit estimates of the convergence
rate and of the time needed to reach a neighborhood of the final equilibrium
\cite{DF1,DF2}. It is also worth mentioning that the reaction-diffusion system
\eqref{RDgen} is actually the {\em gradient flow} of the entropy function
\eqref{Phidef} with respect to an appropriate metric based on the Wasserstein
distance for the diffusion part of the system \cite{LM,Mi1,MHM}. Finally, we
observe that Lyapunov functions such as the entropy \eqref{Phidef} were also
useful to prove global existence of solutions to reaction-diffusion systems, see
\cite{CV,CGV,FMT,Fi,GV,Pi,PSY,So}.

Much less is known on the dynamics of the reaction-diffusion system
\eqref{RDgen} in an unbounded domain such as $\Omega = \R^N$. For bounded
solutions, the entropy \eqref{Phidef} is typically infinite, and it is known
that \eqref{RDgen} is no longer a gradient system.  Solutions such as traveling
waves, which exist in many examples, do not converge to equilibria as
$t \to +\infty$, at least not in the topology of uniform convergence on
$\Omega$. In fact, the best we can hope for in general is {\em
  quasiconvergence}, namely uniform convergence on compact subsets of $\Omega$
to the family of spatially homogeneous equilibria. That property is not
automatic at all, and has been established so far only for relatively simple
scalar equations where the maximum principle is applicable
\cite{DP,MP,PP1,PP2,P1,P2,P3}.  On the other hand, it is important to mention
that entropy is still {\em locally} dissipated under the evolution defined by
\eqref{RDgen}, in the sense that the entropy density $e(x,t)$, the entropy flux
$\fbf(x,t)$ and the entropy dissipation $d(x,t)$ satisfy the local entropy
balance equation $\partial_t e = \div \fbf - d$.  We have the explicit expressions
\begin{equation}\label{edfgen}
\begin{split}
  e(x,t) \,&=\, \sum_{i=1}^n \phi\bigl(c_i(x,t)\bigr)\,, \qquad
  \fbf(x,t) \,=\, \sum_{i=1}^n d_i \log(c_i(x,t))\nabla c_i(x,t)\,, \\
  d(x,t) \,&=\, \sum_{i=1}^n d_i \frac{|\nabla c_i(x,t)|^2}{c_i(x,t)}
  + k \log\biggl(\frac{B(x,t)}{A(x,t)}\biggr)\Bigl(B(x,t) - A(x,t)
  \Bigr)\,,
\end{split}
\end{equation}
from which we deduce the pointwise estimate $|\fbf|^2 \le Ced \log(2+e)$ for
some constant $C > 0$. This precisely means that the reaction-diffusion system
\eqref{RDgen} is an {\em extended dissipative system} in the sense of our
previous work \cite{GS1}. If $N \le 2$, the results of \cite{GS1} show that all
bounded solutions of \eqref{RDgen} in $\R^N$ converge uniformly on compact
subsets to the family of spatially homogeneous equilibria for ``almost all''
times, i.e. for all times outside a subset of $\R_+$ of zero density in the
limit where $t \to +\infty$. In particular, the $\omega$-limit set of any
bounded solution, with respect to the topology of uniform convergence on compact
sets, always contains an equilibrium.  It should be mentioned, however, that
extended dissipative systems in the sense of \cite{GS1} may have
non-quasiconvergent solutions, even in one space dimension. A typical phenomenon
that prevents quasiconvergence is the coarsening dynamics that is observed, for
instance, in the one-dimensional Allen-Cahn equation \cite{ER,P1}.

In the present paper, we consider a very simple particular case of the
reaction-diffusion system \eqref{RDgen}, for which we can prove that all
positive solutions converge uniformly on compact sets to the family of spatially
homogeneous equilibria. In that example we only have two species $\cA$, $\cB$
which participate to the simplistic reaction
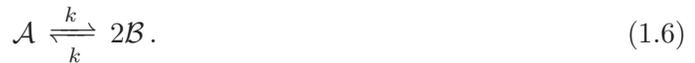
\begin{equation}\label{chem2}
  \cA \,\xrightleftharpoons[~k~]{k}\, 2 \cB\,.
\end{equation}
Denoting by $u,v$ the concentrations of $\cA, \cB$, respectively, 
we obtain the system
\begin{equation}\label{RD}
\begin{split}
  u_t(x,t) \,&=\, a u_{xx}(x,t) + k\bigl(v(x,t)^2-u(x,t)\bigr)\,, \\
  v_t(x,t) \,&=\, b v_{xx}(x,t) + 2k\bigl(u(x,t)-v(x,t)^2\bigr)\,,
\end{split}
\end{equation}
which is considered on the whole real line $\Omega = \R$. The parameters are the
diffusion coefficients $a,b > 0$ and the reaction rate $k > 0$, but scaling
arguments reveal that the ratio $a/b$ is the only relevant quantity. It is not
difficult to verify that, given bounded and nonnegative initial data $u_0, v_0$,
the system \eqref{RD} has a unique global solution that remains bounded and
nonnegative for all positive times, see Proposition~\ref{prop:exist} below for a
precise statement. Our goal is to investigate the long-time behavior of those
solutions, using the local form of the entropy dissipation and some additional
properties of the system.

As a warm-up we consider the case of equal diffusivities $a = b$, which is
considerably simpler because the function $w = 2u + v$ then satisfies the
one-dimensional heat equation $w_t = a w_{xx}$. Using that observation, it is
easy to prove the following result\:

\begin{prop}\label{main1}
If $a = b$ any bounded nonnegative solution of \eqref{RD} satisfies, 
for all $t > 0$, 
\begin{equation}\label{mainest1}
  t\,\|u_x(t)\|_{L^\infty}^2 + t\,\|v_x(t)\|_{L^\infty}^2 + (1+t) 
  \|u(t) - v(t)^2\|_{L^\infty} \,\le\, C\,, 
\end{equation}
where the constant only depends on the parameters $a,k$ and on 
$\|u_0\|_{L^\infty},\|v_0\|_{L^\infty}$. 
\end{prop}

Proposition~\ref{main1} implies that all nonnegative solutions converge, 
uniformly on compact intervals $I \subset \R$, to the manifold of spatially 
homogeneous equilibria defined by
\begin{equation}\label{cEdef}
  \cE \,=\, \Bigl\{(\bar u,\bar v) \in \R_+^2\,;\, \bar u = 
  \bar v^2\Bigr\}\,,
\end{equation}
see Corollary~\ref{cor:ulconv} below for a precise statement. In other words,
the $\omega$-limit set of any solution, with respect to the topology of uniform
convergence on compact sets, is entirely contained in $\cE$.  The proof shows
that the decay rates given by \eqref{mainest1} cannot be improved in general.
Moreover, it is clear that the $\omega$-limit set is not always reduced to a
single equilibrium, because examples of nonconvergent solutions can be
constructed even for the linear heat equation on $\R$, see \cite{CE}.

The proof Proposition~\ref{main1} heavily relies on the simple evolution
equation satisfied by the auxiliary function $w = 2u+v$, which is specific to
the case of equal diffusivities. The analysis becomes much more challenging when
$a \neq b$, because system~\eqref{RD} does not reduce to a scalar equation.  Our
result in the general case is slightly weaker, and can be stated as follows.

\begin{prop}\label{main2}
Any bounded nonnegative solution of \eqref{RD} satisfies, for all $t > 0$, 
\begin{equation}\label{mainest2}
  \|u_x(t)\|_{L^\infty} + \|v_x(t)\|_{L^\infty} \,\le\, 
  \frac{C}{t^{1/2}}\,\log(2+t)\,, \qquad  \|u(t) - v(t)^2\|_{L^\infty} 
  \,\le\, \frac{C}{(1+t)^{1/2}}\,,
\end{equation}
where the constant only depends on the parameters $a,b,k$ 
and on $\|u_0\|_{L^\infty},\|v_0\|_{L^\infty}$. 
\end{prop}

The decay rates of the derivatives $u_x, v_x$ in \eqref{mainest2} agree with
\eqref{mainest1} up to a logarithmic correction, but the estimate of the
difference $u - v^2$, which measures the distance to the local chemical
equilibrium, is substantially weaker in the general case. We conjecture that the
discrepancy between the conclusions of Propositions~\ref{main1} and \ref{main2}
is of technical nature, and that the optimal estimates \eqref{mainest1} remain
valid when $a \neq b$.  At this point, it is worth mentioning that the bounds
\eqref{mainest2} are actually derived from a {\em uniformly local} estimate
which fully agrees with the decay rates given in \eqref{mainest1}. Indeed, we
shall prove in Section~\ref{sec4} that any bounded nonnegative solution to
\eqref{RD} satisfies, for any $t > 0$,
\begin{equation}\label{ulest}
  \sup_{x_0 \in \R}\,\int_{x_0-\sqrt{t}}^{x_0+\sqrt{t}} \Bigl(|u_x(x,t)|^2 + |v_x(x,t)|^2 + 
  \bigl|u(x,t) - v(x,t)^2\bigr|\Bigr) \dd x \,\le\, C t^{-1/2}\,, 
\end{equation}
where the constant depends only on the parameters $a,b,k$ and on the initial
data. It is obvious that \eqref{mainest1} implies \eqref{ulest}, but the 
converse is not quite true and the best we could obtain so far is the weaker 
estimate \eqref{mainest2}. 

As before, we can conclude that all solutions converge uniformly on compact sets
to the manifold $\cE$ as $t \to +\infty$.  

\begin{cor}\label{cor:ulconv}
Under the assumptions of Proposition~\ref{main2}, the solution of \eqref{RD}
satisfies, for any time $t > 0$ and any bounded interval $I \subset \R$, 
\begin{equation}\label{ulconv}
  \inf\Bigl\{\|u(t) - \bar u\|_{L^\infty(I)} + \|v(t) - \bar v\|_{L^\infty(I)}
  \,;\, (\bar u,\bar v) \in \cE\Bigr\} \,\le\, \frac{C |I|}{|I| + t^{1/2}}
  \,\log(2+t)\,,
\end{equation}
where the constant only depends on the parameters $a,b,k$ and on 
$\|u_0\|_{L^\infty},\|v_0\|_{L^\infty}$. 
\end{cor}

\begin{rem}\label{rem:positive}
It is important to keep in mind that the conclusions of Propositions~\ref{main1}
and \ref{main2} are restricted to nonnegative solutions. As a matter of 
fact, the dynamics of system~\eqref{RD} is completely different if we 
consider solutions for which the second component $v$ may take negative
values. For instance, if $a = b = k = 1$, we can look for solutions of 
the particular form
\[
  u(x,t) \,=\, 1 - \frac{3 z(x,t)}{4}\,, \qquad
  v(x,t) \,=\, - 1 + \frac{3 z(x,t)}{2}\,,
\]
in which case \eqref{RD} reduces to the Fisher-KPP equation $z_t = z_{xx} 
+ 3z(1-z)$. That equation has a pulse-like stationary solution given 
by the explicit formula
\[
  \bar z(x) \,=\, 1 - \frac{3}{2}\,\frac{1}{\cosh^2(\sqrt{3}x/2)}\,, \qquad
  x \in \R\,, 
\]
which provides an example of a steady state $(\bar u,\bar v)$ for \eqref{RD}
that is not spatially homogeneous nor at chemical equilibrium, in the sense
that $\bar u \neq \bar v^2$. Moreover, for any speed $c > 0$, the Fisher-KPP 
equation has traveling wave solutions of the form $z(x,t) = \varphi(x-ct)$ 
where the wave profile $\varphi$ satisfies $\varphi(-\infty) = 1$ and 
$\varphi(+\infty) = 0$.  For the corresponding solutions of \eqref{RD}, the 
quantities $\|u_x(t)\|_{L^\infty}$, $\|v_x(t)\|_{L^\infty}$, and $\|u(x) - 
v(t)^2\|_{L^\infty}$ are bounded away from zero for all times, in sharp 
contrast with \eqref{mainest1}.  
\end{rem}

\begin{rem}\label{rem:bounded}
Our results also apply to the situation where system~\eqref{RD}
is considered on a bounded interval $I = [0,L]$, with homogeneous Neumann 
boundary conditions, because the solutions $u,v$ can then be extended to 
even and $2L$-periodic functions on the whole real line. In that case
the total mass $M = \int_0^L \bigl(2u(x,t) + v(x,t)\bigr)\dd x$ is a 
conserved quantity, and the solution necessarily converges to the unique 
equilibrium $(u_\infty,v_\infty) \in \cE$ satisfying $2u_\infty + v_\infty = M/L$. 
As in \eqref{ulconv} we have the bound
\[
  \|u(t) - u_\infty\|_{L^\infty(I)} + \|v(t) - v_\infty\|_{L^\infty(I)}
  \,\le\, \frac{C L}{L + t^{1/2}}\,\log(2+t)\,, \qquad t \ge 0\,,
\]
which is far from optimal because, in that particular case, it is known that
convergence occurs at exponential rate, see \cite{DF1} for accurate estimates
with explicitly computable constants. However, the conclusion of
Proposition~\ref{main2} remains interesting in that situation. In particular,
the second estimate in \eqref{mainest2} shows that the time needed for a
solution to enter a neighborhood of the manifold $\cE$ depends on the $L^\infty$
norm of the initial data, but {\em not} on the length $L$ of the interval. In
contrast, all estimates obtained in \cite{DF1} and related works necessarily
involve the size of the spatial domain, because they use as a Lyapunov function 
the total entropy which is an extensive quantity in the thermodynamical sense.
\end{rem}

The proof of our main result, Proposition~\ref{main2}, is based on localized
energy (or entropy) estimates in the spirit of our previous works
\cite{GS0,GS1,GS2}. It turns out that the Boltzmann-type entropy density
introduced in \eqref{edfgen} is not the only possibility. Quite on the contrary,
there exist a large family of nonnegative quantities that are locally dissipated
under the evolution defined by the two-component system \eqref{RD}, see
Section~\ref{sec3} below for a more precise discussion. For simplicity, we chose 
to use the energy density $e(x,t)$, the energy flux $f(x,t)$, and the energy 
dissipation $d(x,t)$ given by the following expressions:
\begin{equation}\label{edf2x2}
  e \,=\, \frac{1}{2}\,u^2 + \frac{1}{6}\,v^3\,, \qquad
  f \,=\, a u u_x + \frac{b}{2}\,v^2 v_x\,, \qquad 
  d \,=\, a u_x^2 + b v v_x^2 + k(u-v^2)^2\,.
\end{equation}
If $(u,v)$ is any nonnegative solution of \eqref{RD}, one readily verifies that
the local energy balance $\partial_t e = \partial_x f - d$ is satisfied, as well
as the estimate $f^2 \le C e d$ where $C = \max(2a,3b/2)$.  Altogether, this
means that \eqref{RD} is an ``extended dissipative system'' in the sense of
\cite{GS1}. As was already mentioned, the results of \cite{GS1} provide useful
information on the gradient-like dynamics of \eqref{RD}, but this is far from
sufficient to prove Proposition~\ref{main2}.  For instance, extended dissipative
systems may have traveling wave solutions which, obviously, do not satisfy
uniform decay estimates of the form \eqref{mainest2}.

To go beyond the general results established in \cite{GS1} we follow the same
approach as in our previous work \cite{GS2}, where energy methods were developed
to study the long-time behavior of solutions for the Navier-Stokes equations in
the infinite cylinder $\R \times \T$. The main idea is to show that the energy
dissipation in \eqref{edf2x2} is itself locally dissipated under the evolution
defined by \eqref{RD}. More precisely, we look for another triple $(\tilde e,
\tilde f, \tilde d)$ satisfying the local balance $\partial_t \tilde e = 
\partial_x \tilde f - \tilde d$, and such that the flux $|\tilde f|$ can be 
controlled in terms of $\tilde e$, $\tilde d$.  We also require that
$\tilde d \ge 0$ and that $\tilde e \approx d$, where $d$ is as in
\eqref{edf2x2}. We can then use localized energy estimates as in \cite{GS1,GS2}
to prove that, on any compact interval $[x_0,x_0+L] \subset \R$, the dissipation
$d(x,t)$ becomes uniformly small for all times $t \gg L^2$.  We even get an
explicit upper bound depending only on $L$ and on the initial data, so that
taking the supremum over $x_0 \in \R$ we arrive at estimate \eqref{ulest}, which
is the crucial step in the proof of Proposition~\ref{main2}.  In contrast, we
emphasize that the bounds one can obtain using the dissipative structure
\eqref{edf2x2} alone only show that the supremum of $d(x,t)$ over $[x_0,x_0+L]$
becomes small for ``almost all'' (sufficiently large) times, thus leaving space
for non-gradient transient behaviors such as traveling wave propagation or
coarsening dynamics.

The existence of a second dissipative structure on top of \eqref{edf2x2} is
obviously an important property of system~\eqref{RD}, which we would like to
understand in greater depth. It should be related to some convexity property of
the energy density with respect to the metric that turns \eqref{RD} into a
gradient system, see \cite{LM} for a more detailed discussion of gradient
structures and convexity properties of reaction-diffusion systems. It would be
interesting to determine if that property still holds for other systems of the
form \eqref{RDgen}, such as those considered in Section~\ref{sec6} below, but so
far we have no general result in that direction. We mention that the idea of
studying the variation of the entropy dissipation, or equivalently the second
variation of the entropy, is quite common in kinetic theory, see \cite{DV}, 
as well as in fluid mechanics, see \cite{AB} for a recent review on the 
subject. 

The rest of this paper is organized as follows. In Section~\ref{sec2} we briefly
discuss the Cauchy problem for the reaction-diffusion system \eqref{RD}, and we
prove Proposition~\ref{main1} and Corollary~\ref{cor:ulconv}. After these
preliminaries, we investigate in Section~\ref{sec3} various dissipative
structures of the form \eqref{edf2x2}, which play a key role in our analysis. 
The proof of Proposition~\ref{main2} is completed in Section~\ref{sec4}, where
we use localized energy estimates inspired from our previous works \cite{GS1,GS2}. 
Section~\ref{sec5} is devoted to the stability analysis of spatially homogeneous
equilibria $(\bar u,\bar v) \in \cE$, which provides useful insight on the 
decay rates of the solutions. In the final Section~\ref{sec6}, we briefly 
discuss the potential applicability of our method to more general 
reaction-diffusion systems of the form \eqref{RDgen}, and we mention some
open problems. 

\medskip\noindent{\bf Acknowledgments.}
The authors thank Alexander Mielke for enlightening discussions at the early
stage of this project. Th.G. is supported by the grant ISDEEC ANR-16-CE40-0013
of the French Ministry of Higher Education, Research and Innovation. 
S.S. is supported by the Croatian Science Foundation under grant 
IP-2018-01-7491. 

\section{Preliminary results}\label{sec2}

We first prove that system~\eqref{RD} is globally well-posed for all initial data
$u_0$, $v_0$ that are bounded and nonnegative.  This a rather classical
statement, which can be deduced from more general results on reaction-diffusion
systems with quadratic nonlinearities, see e.g. \cite{GV,Pi,So}. For the
reader's convenience, we give here a simple and self-contained proof.

Without loss of generality, we assume henceforth that $k = 1$. We denote by
$X = C_\bu(\R)$ the Banach space of all bounded and uniformly continuous
functions $f : \R \to \R$, equipped with the uniform norm $\|f\|_{L^\infty}$. 
Since we are interested in nonnegative solutions of \eqref{RD}, we also 
define the positive cone $X_+ = \{f \in X\,;\, f(x) \ge 0~\forall x \in \R\}$.

\begin{prop}\label{prop:exist}
For all initial data $(u_0,v_0) \in X_+^2$, system~\eqref{RD} has a unique global 
(mild) solution $(u,v) \in C^0([0,+\infty),X^2)$ such that $(u(0),v(0)) = (u_0,v_0)$. 
Moreover $(u(t),v(t)) \in X_+^2$ for all $t \ge 0$, and the following estimates hold\:
\begin{equation}\label{uvbound}
\begin{split}
  \max\bigl(\|u(t)\|_{L^\infty}\,,\,\|v(t)\|_{L^\infty}^2\bigr) 
  \,&\le\, \max\bigl(\|u_0\|_{L^\infty}\,,\,\|v_0\|_{L^\infty}^2\bigr)\,, \\
  2\|u(t)\|_{L^\infty} + \|v(t)\|_{L^\infty} \,&\le\,  
  2\|u_0\|_{L^\infty} + \|v_0\|_{L^\infty}\,. 
\end{split}
\end{equation}
\end{prop} 

\begin{proof}
Local existence of solutions in $X^2$ can be established by applying a standard 
fixed point argument to the integral equation
\[
  \begin{pmatrix} u(t) \\ v(t)\end{pmatrix} \,=\, 
  \begin{pmatrix} S(at) & 0 \\ 0 & S(bt)\end{pmatrix}
  \begin{pmatrix} u_0 \\ v_0\end{pmatrix} +   
  \int_0^t \begin{pmatrix} S(a(t{-}s)) & 0 \\ 0 & S(b(t{-}s))\end{pmatrix}
  \begin{pmatrix} v(s)^2 - u(s) \\ 2\bigl(u(s) - v(s)^2\bigr)\end{pmatrix}\dd s\,, 
\]
where $S(t) = \exp(t\partial_x^2)$ is the one-dimensional heat semigroup, see
e.g. \cite[Chapter~3]{He}. Since the nonlinearity is a polynomial of degree two,
the local existence time $T > 0$ given by the fixed point argument is no smaller
than $T_0\bigl(1+\|u_0\|_{L^\infty} +\|v_0\|_{L^\infty}\bigr)^{-1}$ for some
constant $T_0 > 0$. This shows that any local solution can be extended to a
global one, unless the quantity $\|u(t)\|_{L^\infty} + \|v(t)\|_{L^\infty}$
blows up in finite time. It remains to show that nonnegative solutions 
satisfy the estimates \eqref{uvbound}, so that blow-up cannot occur. 

Assume that $(u_0,v_0) \in X_+^2$ and let $(u,v) \in C^0([0,T_*),X^2)$ be 
the maximal solution of \eqref{RD} with initial data $(u_0,v_0)$. This 
solution is smooth for positive times, and the first component satisfies
$u_t = a u_{xx} + v^2 - u \ge a u_{xx} - u$ for $t \in (0,T_*)$. Applying 
the parabolic maximum principle \cite{PW}, we deduce that $u(t) \in X_+$ 
for all $t \in (0,T_*)$. The second component in turn satisfies $v_t = 
b v_{xx} + 2(u - v^2) \ge b v_{xx} - 2v^2$, and another application of 
the maximum principle shows that $v(t) \in X_+$ too. So the positive 
cone $X_+^2$ is invariant under the evolution defined by \eqref{RD}. 
  
Another important observation is that \eqref{RD} is a {\em cooperative} 
reaction-diffusion system in $X_+^2$, in the sense that the reaction terms 
in \eqref{RD} satisfy
\[
  \frac{\D}{\D v} \bigl(v^2 - u\bigr) \,=\, 2v \,\ge\, 0\,, \qquad
  \frac{\D}{\D u} \,2\bigl(u - v^2\bigr) \,=\, 2 \,\ge\, 0\,.
\]
As is well known, such a system obeys a (component-wise) comparison principle
\cite{VVV}. In our case, this means that, if $(u,v)$ and $(\bar u,\bar v)$ are
two solutions of \eqref{RD} in $X_+^2$, and if the initial data satisfy
$u_0 \le \bar u_0$ and $v_0 \le \bar v_0$, then $u(t) \le \bar u(t)$ and
$v(t) \le \bar v(t)$ as long as the solutions are defined.  We use that
principle to compare our nonnegative solution $(u,v)$ to the solution
$(\bar u,\bar v)$ of the ODE system 
\begin{equation}\label{ODE}
  \frac{\D}{\D t}\,\bar u(t) \,=\, \bar v(t)^2 - \bar u(t)\,, \qquad
  \frac{\D}{\D t}\,\bar v(t) \,=\, 2 \bigl(\bar u(t) - \bar v(t)^2\bigr)\,,
\end{equation}
with initial data $\bar u_0 = \|u_0\|_{L^\infty}$, $\bar v_0 = \|v_0\|_{L^\infty}$. 
The dynamics of \eqref{ODE} in the positive quadrant is very simple\: 
the solution stays on the line $L_0 = \bigl\{(\bar u,\bar v) \in \R_+^2 \,;\, 
2\bar u + \bar v = 2\bar u_0 + \bar v_0\bigr\}$ for all times, and converges 
to the unique equilibrium $(\bar u_*,\bar v_*) \in L_0 \cap \cE$, where 
$\cE$ is defined in \eqref{cEdef}; see Figure~1. In particular, we have
$\max(\bar u(t),\bar v(t)^2) \le \max(\bar u_0,\bar v_0^2)$ for all 
$t \ge 0$. Applying the comparison principle, we conclude that our 
solution $(u,v) \in C^0([0,T_*),X^2)$ satisfies estimates \eqref{uvbound}
for all $t \in [0,T_*)$, which implies that $T_* = + \infty$. 
\end{proof}

\begin{rem}\label{rem:ODE}
The equilibrium $(\bar u_*,\bar v_*)$ which attracts the solution of 
\eqref{ODE} is given by
\begin{equation}\label{uvstar}
  \bar u_* \,=\, \bar v_*^2\,, \qquad \hbox{and}\qquad 
  \bar v_* \,=\, \frac{1}{4}\Bigl(-1 + \sqrt{1 + 16 \bar u_0 + 8 \bar v_0}\Bigr)\,.
\end{equation}
As is clear from Figure~1, we have the optimal bounds
\begin{equation}\label{uvbdd}
  \min\bigl(\bar u_0, \bar u_*\bigr) \,\le\, \bar u(t) \,\le\, 
  \max\bigl(\bar u_0, \bar u_*\bigr)\,, \qquad
  \min\bigl(\bar v_0, \bar v_*\bigr) \,\le\, \bar v(t) \,\le\, 
  \max\bigl(\bar v_0, \bar v_*\bigr)\,,
\end{equation}
which can be used to improve somewhat \eqref{uvbound}. 
\end{rem}

\begin{rem}\label{rem:lowerbd}
In a similar way, we can use the comparison principle to show that the
solution of \eqref{RD} given by Proposition~\ref{prop:exist} satisfies 
$u(x,t) \ge \underline{u}(t)$ and $v(x,t) \ge \underline{v}(t)$, where 
$(\underline{u}(t),\underline{v}(t))$ is the solution of the ODE system
\eqref{ODE} with initial data
\[
  \underline{u}_0 \,=\, \inf_{x \in \R} u_0(x)\,, \qquad
  \underline{v}_0 \,=\, \inf_{x \in \R} v_0(x)\,. 
\]
Two interesting conclusions can be drawn using such lower bounds. First, if 
$\underline{v}_0 \ge \delta > 0$ for some $\delta > 0$, then $v(x,t) \ge 
2\delta\bigl(1+\sqrt{1+8\delta}\bigr)^{-1}$ for all $x \in \R$ and all $t \ge 0$. 
This observation will be used in the proof of Proposition~\ref{main2}. Second, 
any homogeneous equilibrium $(\bar u,\bar v) \in \cE$ is stable (in the 
sense of Lyapunov) in the uniform topology: for any $\epsilon > 0$, 
there exists $\delta > 0$ such that, if $\|u_0 - u_*\|_{L^\infty} + 
\|v_0 - v_*\|_{L^\infty} \le \delta$, then $\|u(t) - u_*\|_{L^\infty} + 
\|v(t) - v_*\|_{L^\infty} \le \epsilon$ for all $t \ge 0$. An explicit 
expression for $\delta$ in terms of $\epsilon$ and $u_*,v_*$ can be 
deduced from \eqref{uvstar}, \eqref{uvbdd}. 
\end{rem}

\begin{rem}\label{rem:Linfty}
In Proposition~\ref{prop:exist} we assume for simplicity that the initial
data $u_0, v_0$ are bounded and uniformly continuous, but system~\eqref{RD} 
remains globally well posed for all nonnegative data $(u_0,v_0) \in L^\infty(\R)^2$. 
The only difference in the proof is that, when $t \to 0$, the first term in the 
integral equation does not converge to $(u_0,v_0)$ in the uniform norm, but only 
in the weak-$*$ topology of $L^\infty(\R)$.  
\end{rem}

\setlength{\unitlength}{0.8cm}
\begin{center}
\begin{picture}(10,8)(-1,-1)
%% axes and curve
\thicklines
\put(-0.5,0){\vector(1,0){10.0}}
\put(0,-0.5){\vector(0,1){7.0}}
\qbezier(0,0)(0,3)(9,6)
%% straight lines
\thinlines
%\put(0.5,3){\vector(1,-1){2.5}}
\put(0.5,3){\vector(1,-1){0.5}}
\put(1.0,2.5){\line(1,-1){0.5}}
\put(3.0,0.5){\vector(-1,1){1.5}}
%\put(0.5,5){\line(1,-1){4.5}}
\put(0.5,5){\vector(1,-1){1.6}}
\put(2.1,3.4){\line(1,-1){0.6}}
\put(5.0,0.5){\vector(-1,1){2.3}}
%\put(1.5,6){\line(1,-1){5.5}}
\put(1.5,6){\vector(1,-1){1.8}}
\put(3.3,4.2){\line(1,-1){0.7}}
\put(7.0,0.5){\vector(-1,1){3.0}}
%\put(3.5,6){\line(1,-1){4.5}}
\put(3.5,6){\vector(1,-1){1.2}}
\put(4.7,4.8){\line(1,-1){0.6}}
\put(8.0,1.5){\vector(-1,1){2.7}}
%\put(5.5,6){\line(1,-1){2.5}}
\put(5.5,6){\vector(1,-1){0.6}}
\put(6.1,5.4){\line(1,-1){0.6}}
\put(8.0,3.5){\vector(-1,1){1.3}}
%% objects and legends
\put(9,0.3){$u$}
\put(0.3,6.1){$v$}
\put(8.6,5.4){$\cE$}
\put(2.3,5){\footnotesize{$\bullet$}}
\put(3.57,3.73){\footnotesize{$\bullet$}}
\put(2.3,5.3){\footnotesize{$(u_0,v_0)$}}
\put(5.2,2.5){\footnotesize{$L_0$}}
\put(3.65,-0.2){\line(0,1){0.4}}
\put(-0.2,3.85){\line(1,0){0.4}}
\put(3.79,-0.4){$u_*$}
\put(-0.7,3.8){$v_*$}
\end{picture}
\end{center}

\begin{center}
\begin{minipage}[c]{0.8\textwidth}\footnotesize
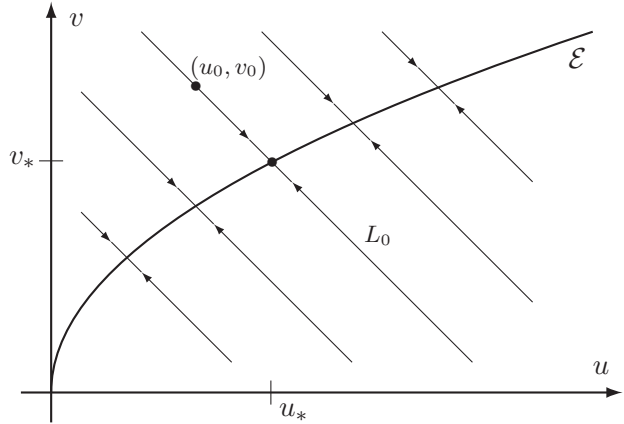
{\bf Figure~1\:} A sketch of the dynamics of the ODE system $\dot u = v^2 - u$, 
$\dot v = 2(u - v^2)$, which represents the kinetic part of \eqref{RD}. The
solution starting from the initial data $(u_0,v_0)$ stays on the line 
$L_0 = \{(u,v)\,;\, 2u + v = 2u_0 + v_0\}$ and converges there to the 
unique equilibrium $(u_*,v_*) \in L_0\cap \cE$. 
\end{minipage}
\end{center}
\medskip

\noindent{\bf Proof of Proposition~\ref{main1}.}
We assume here without loss of generality that $a = b = k = 1$. Given 
$(u_0,v_0) \in X_+^2$, let $(u,v) \in C^0([0,+\infty),X^2)$ be the 
unique global solution of \eqref{RD} with initial data $(u_0,v_0)$. 
As was already mentioned, the quantity $w = 2u+v$ satisfies the linear
heat equation $w_t = w_{xx}$ on $\R$. In particular, we have the estimate
\begin{equation}\label{wder}
  \|w_x(t)\| \,\le\, \frac{C\|w_0\|}{t^{1/2}} \,\le\, \frac{CR}{t^{1/2}}\,, 
  \qquad t > 0\,,
\end{equation}
where $R := 1 + \|u_0\| + \|v_0\|$. Here and in what follows, we denote 
$\|\cdot\| = \|\cdot\|_{L^\infty}$, and the generic constant $C$ is always
independent of the initial data $(u_0,v_0)$. 

We first estimates the derivatives $u_x(t), v_x(t)$ for $t \le t_0$, where 
$t_0 := T_0/R$ is the local existence time appearing in the proof of 
Proposition~\ref{prop:exist}. Differentiating the integral equation
and using the second inequality in \eqref{uvbound}, we easily obtain
\begin{equation}\label{uvder}
  \|u_x(t)\| + \|v_x(t)\| \,\le\, \frac{CR}{t^{1/2}} + \int_0^t 
  \frac{CR^2}{(t-s)^{1/2}}\dd s \,\le\, \frac{CR}{t^{1/2}}\,,
  \qquad 0 < t \le t_0\,.
\end{equation}
In particular, we have $\|u_x(t_0)\| + \|v_x(t_0)\| \le C R^{3/2}$. 

We next observe that the quantity $q = v_x$ satisfies the equation $q_t = q_{xx} 
- (1+4v)q + w_x$. The corresponding integral equation reads
\[
  q(t) \,=\, \Sigma(t,t_0)q(t_0) + \int_{t_0}^t \Sigma(t,s) 
  w_x(s) \dd s\,, \qquad t > t_0\,,
\]
where $\Sigma(t,s)$ is the two-parameter semigroup associated with the 
linear nonautonomous equation $q_t = q_{xx} - (1+4v)q$, assuming that 
$v(x,t)$ is given. Since $v \ge 0$, the maximum principle implies
the pointwise estimate $\Sigma(t,s) \le e^{-(t-s)} S(t-s)$, where $S(t) = 
\exp(t\partial_x^2)$ is the heat kernel. Using \eqref{wder}, \eqref{uvder},
we thus obtain 
\begin{equation}\label{qest}
\begin{split}
  \|q(t)\| \,&\le\, e^{-(t-t_0)}\|q(t_0)\| + \int_{t_0}^t e^{-(t-s)}\|w_x(s)\|\dd s \\
  \,&\le\, C R^{3/2}\,e^{-(t-t_0)} + \int_{t_0}^t e^{-(t-s)}\,\frac{CR}{s^{1/2}}\dd s
  \,\le\, \frac{CR^{3/2}}{t^{1/2}}\,, \qquad t > t_0\,.
\end{split}
\end{equation}
Note that \eqref{uvder}, \eqref{qest} imply that $\|q(t)\| \le C R^{3/2}t^{-1/2}$ 
for all $t > 0$. 

Similarly, the quantity $p = u_x$ satisfies the equation $p_t = p_{xx} 
- p + 2vq$, and we know from \eqref{uvbound} that $\|v(t)\| \le 2R$ for all 
$t \ge 0$. It follows that
\begin{equation}\label{pest}
\begin{split}
  \|p(t)\| \,&\le\, e^{-(t-t_0)}\|p(t_0)\| + 4R \int_{t_0}^t e^{-(t-s)}\|q(s)\|\dd s \\
  \,&\le\, C R^{3/2}\,e^{-(t-t_0)} + C \int_{t_0}^t e^{-(t-s)}\,\frac{R^{5/2}}{s^{1/2}}\dd s
  \,\le\, \frac{CR^{5/2}}{t^{1/2}}\,, \qquad t > t_0\,.
\end{split}
\end{equation}
Altogether we deduce from \eqref{uvder}, \eqref{qest}, \eqref{pest} that 
$t \|u_x(t)\|^2 + t \|v_x(t)\|^2 \le C R^5$ for all $t > 0$, which proves
the first inequality in \eqref{mainest1}. 

Finally, the quantity $\rho = u - v^2$ satisfies the equation $\rho_t = 
\rho_{xx} - (1+4v)\rho + 2q^2$ as well as the a priori estimate $\|\rho(t)\|
\le R^2$ for all $t \ge 0$. Proceeding as above and using \eqref{qest}, we 
find
\begin{equation}\label{rhoest}
\begin{split}
  \|\rho(t)\| \,&\le\, e^{-(t-t_0)}\|\rho(t_0)\| + 2\int_{t_0}^t e^{-(t-s)}\|q(s)\|^2\dd s \\
  \,&\le\, R^2\,e^{-(t-t_0)} + C \int_{t_0}^t e^{-(t-s)}\,\frac{R^3}{s}\dd s
  \,\le\, \frac{CR^3}{t}\,\log(1{+}R)\,, \qquad t > t_0\,.
\end{split}
\end{equation}
Thus $(1+t)\|\rho(t)\| \le C R^3\log(1+R)$ for all $t \ge 0$, which 
concludes the proof of \eqref{mainest1}. \QED

\begin{rem}\label{rem:higher}
Similarly, differentiating with respect to $x$ the evolution equations 
for the quantities $q,p,\rho$ and using an induction argument, one can 
show that the solution of \eqref{RD} with $a = b$ satisfies, for each 
integer $m \in \N$, an estimate of the form
\begin{equation}\label{higherest}
  \|\partial_x^m u(t)\|_{L^\infty} + \|\partial_x^m v(t)\|_{L^\infty} \,\le\, 
  \frac{C_m}{t^{m/2}}\,, \qquad \|\partial_x^m \rho(t)\|_{L^\infty} 
  \,\le\, \frac{C_m}{t^{m/2}(1{+}t)}\,, \qquad \forall\,t > 0\,.
\end{equation}
\end{rem}

\bigskip\noindent{\bf Proof of Corollary~\ref{cor:ulconv}.}
If $t \ge |I|^2$, we pick $x_0 \in I$ and define $\bar v = v(x_0,t)$, 
$\bar u = \bar v^2$. Then $(\bar u,\bar v) \in \cE$ and using the first 
inequality in \eqref{mainest2} we find 
\[
  \|v(\cdot,t) - \bar v\|_{L^\infty(I)} \,=\, \|v(\cdot,t) - v(x_0,t)\|_{L^\infty(I)}
  \,\le\, |I|\,\|v_x(\cdot,t)\|_{L^\infty(I)} \,\le\, \frac{C|I|}{t^{1/2}}\,
  \log(2+t)\,.
\]
Similarly, using in addition the second inequality in \eqref{mainest2}, 
we obtain
\[
  \|u(\cdot,t) - \bar u\|_{L^\infty(I)} \,\le\, \|u(\cdot,t) - u(x_0,t)\|_{L^\infty(I)}
  + |I|\,|u(x_0,t) - v(x_0,t)^2| \,\le\, \frac{C|I|}{t^{1/2}}\,
  \log(2+t)\,.
\]
Combining these bounds and recalling that $t \ge |I|^2$, we arrive at 
\eqref{ulconv}. If $t \le |I|^2$, we can take $\bar u = \bar v = 0$
and use the second bound in \eqref{uvbound} to arrive directly at 
\eqref{ulconv} (without logarithmic correction in that case). 
\QED

\section{An ordered pair of dissipative structures}\label{sec3}

We now relax the assumption that $a = b$ and return to the general 
case where $a,b$ are arbitrary positive constants. Assuming without loss 
of generality that $k = 1$, we write system \eqref{RD} in the equivalent
form
\begin{equation}\label{RD2}
  u_t \,=\, a u_{xx} -\rho\,, \qquad   v_t \,=\, b u_{xx} +2\rho\,,
\end{equation}
where the auxiliary quantity $\rho = u - v^2$ measures the distance to the
chemical equilibrium. 

As was already mentioned, the proof of Proposition~\ref{main2} relies on local
energy estimates and follows the general approach described in \cite{GS1}.
For the nonnegative solutions of \eqref{RD2} given by Proposition~\ref{prop:exist}, 
it is convenient to use the energy density $e(x,t)$, the energy flux $f(x,t)$, 
and the energy dissipation $d(x,t)$ given by \eqref{edf2x2}, namely
\begin{equation}\label{edf1}
  e \,=\, \frac{1}{2}\,u^2 + \frac{1}{6}\,v^3\,, \qquad
  f \,=\, a u u_x + \frac{b}{2}\,v^2 v_x\,, \qquad 
  d \,=\, a u_x^2 + b v v_x^2 + \rho^2\,.
\end{equation}
The local energy balance $\partial_t e = \partial_x f - d$ is easily verified
by a direct calculation. The main properties we shall use are the positivity of
the energy $e$ and the dissipation $d$, as well as the pointwise estimate
$f^2 \le C e d$, where $C > 0$ depends only on $a,b$. In \cite{GS1}, an evolution
equation equipped with a triple $(e,f,d)$ satisfying the above properties is
called an ``extended dissipative system''. According to that terminology, we
shall refer to the triple \eqref{edf1} as an ``EDS structure'' for system
\eqref{RD2}.

The essential step in the proof of Proposition~\ref{main2} is the construction 
of a second EDS structure $(\tilde e,\tilde f,\tilde d)$ for \eqref{RD2}, 
where the new energy density $\tilde e$ is bounded from above by a multiple 
of the energy dissipation $d$ in the first EDS structure. It is quite natural
to look for $\tilde e$ as a linear combination of the quantities $u_x^2$, 
$v v_x^2$, and $\rho^2$ that appear in the expression of $d$ in \eqref{edf1}. 

\begin{lem}\label{lem:edf2}
For all values of the parameters $\alpha,\beta > 0$ the quantities 
$\tilde e$, $\tilde f$, $\tilde d$ defined by 
\begin{equation}\label{edf2}
\begin{split}
  \tilde e \,&=\, \frac{\alpha}{2}\,u_x^2 + \frac{\beta}{2}\,v v_x^2 + 
  \frac{1}{2}\,\rho^2\,, \\ 
  \tilde f \,&=\, \alpha u_x u_t + \beta v v_x v_t - \frac{\beta b}{6}\,v_x^3
  + \frac{\beta}{2}\,\rho \rho_x\,, \\
  \tilde d \,&=\, \alpha a u_{xx}^2 + \beta b v v_{xx}^2 + (1+4v)\rho^2 
  + \frac{\beta}{2}\,\rho_x^2 -\bigl(a + \alpha - \beta/2\bigr)\rho u_{xx} 
  + 2\bigl(b + \beta/2\bigr)\rho v v_{xx}\,,
\end{split}
\end{equation}
satisfy the local energy balance $\partial_t \tilde e = \partial_x \tilde f 
- \tilde d$. 
\end{lem}

\begin{proof}
Differentiating $\tilde e$ with respect to time and using \eqref{RD2}, 
we find by a direct calculation
\begin{equation}\label{prob}
\begin{split}
  \partial_t \tilde e \,&=\, \alpha u_x u_{xt} + \beta v v_x v_{xt} + 
  \frac{\beta}{2}\,v_x^2 v_t + \rho \rho_t \\ 
  \,&=\, \bigl(\alpha u_x u_t + \beta v v_x v_t\bigr)_x - \alpha u_{xx} u_t 
  - \beta v v_{xx} v_t - \frac{\beta}{2}\,v_x^2 v_t + \rho \rho_t \\ 
  \,&=\, \Bigl(\alpha u_x u_t + \beta v v_x v_t - \frac{\beta b}{6}\,v_x^3\Bigr)_x
  -\alpha a u_{xx}^2 - \beta b v v_{xx}^2 - (1+4v)\rho^2  \\[1mm] 
  & \hspace{20pt} - \beta \rho v_x^2 + (a + \alpha)\rho u_{xx} -2 (b + \beta)
  \rho v v_{xx}\,.
\end{split}
\end{equation}
The last line collects the terms which have no definite sign and cannot be 
incorporated in the flux $\tilde f$. Among them, the terms involving 
$\rho u_{xx}$ and $\rho v v_{xx}$ can be controlled by the negative terms in 
the previous line. This is not the case, however, of the term $-\beta \rho v_x^2$, 
which is potentially problematic. The trick here is to use the identity 
\begin{equation}\label{rhoid}
  \rho v_x^2 \,=\, \frac{\rho}{2}\,\bigl(u_{xx} - 2v v_{xx} - \rho_{xx}\bigr)\,, 
\end{equation}
which is easily obtained by differentiating twice the relation $\rho = u - v^2$
with respect to $x$. If we replace \eqref{rhoid} into \eqref{prob} and if we
observe in addition that $\rho \rho_{xx} = (\rho \rho_x)_x - \rho_x^2$, we
conclude that $\partial_t \tilde e = \partial_x \tilde f - \tilde d$, where
$\tilde e$, $\tilde f$, $\tilde d$ are defined in \eqref{edf2}.
\end{proof}

It remains to chose the free parameters $\alpha,\beta$ so that the dissipation
$\tilde d$ is positive. In the third line of \eqref{edf2}, the last two terms
involving $\rho u_{xx}$ and $\rho v v_{xx}$ have no definite sign, but (as 
already mentioned) we can use Young's inequality to control them in terms of 
the positive quantities $u_{xx}^2$, $v v_{xx}^2$, and $(1+4v)\rho^2$. This procedure 
works if and only if
\begin{equation}\label{alphabet}
  \bigl(a + \alpha - \beta/2\bigr)^2 \,<\, 4 a \alpha\,, \qquad \hbox{and}
  \qquad (b + \beta/2\bigr)^2 \,<\, 4 b \beta\,.
\end{equation}
It is always possible to chose $\alpha,\beta > 0$ so that both inequalities 
in \eqref{alphabet} are satisfied. A particularly simple solution is 
$\alpha = a + b$, $\beta = 2b$, which we assume henceforth. We thus find\:

\begin{cor}\label{lem:dpos}
The quantities $\tilde e$, $\tilde f$, $\tilde d$ defined by 
\begin{equation}\label{edf2bis}
\begin{split}
  \tilde e \,&=\, \frac{a+b}{2}\,u_x^2 + b\,v v_x^2 + \frac{1}{2}\,\rho^2\,, \\[-0.5mm]
  \tilde f \,&=\, (a+b) u_x u_t + 2b\,v v_x v_t - \frac{b^2}{3}\,v_x^3
  + b\,\rho \rho_x\,, \\[1mm]
  \tilde d \,&=\, a(a+b) u_{xx}^2 + 2 b^2 v v_{xx}^2 + (1+4v)\rho^2 
  + b\,\rho_x^2 -2a\rho u_{xx} + 4b\rho v v_{xx}\,,
\end{split}
\end{equation}
satisfy the local energy balance $\partial_t \tilde e = \partial_x \tilde f - 
\tilde d$, and there exists a constant $\gamma > 0$ depending only on $a,b$ such 
that
\begin{equation}\label{dpos}
  \tilde d \,\ge\, \tilde d_0 \,:=\, \gamma\Bigl(u_{xx}^2 + v v_{xx}^2 + 
  (1+4v)\rho^2\Bigr) + b\rho_x^2\,.
\end{equation}
\end{cor}

\begin{rem}\label{rem:flux}
Strictly speaking, the triple $(\tilde e,\tilde f,\tilde d)$ is not an EDS
structure in the sense of \cite{GS1}, because the flux bound $\tilde f^2 
\le C \tilde e \tilde d$ does not hold. The problem comes from the term 
involving $v_x^3$ in $\tilde f$\: it is clearly not possible to bound $v_x^6$ 
pointwise in terms of $v v_x^2$ and $v v_{xx}^2$. Nevertheless we shall see in
Section~\ref{sec4} that the contribution of that term to the localized
energy estimates can be estimated as if the pointwise bound was valid. 
This suggests that our definition of ``extended dissipative system'' given 
in \cite{GS1} may be too restrictive, and should perhaps be generalized so as 
to include more general flux terms as in the present example.  
\end{rem}

The EDS structures \eqref{edf1}, \eqref{edf2bis} provide a good control on the
quantities $v^2$, $v_x^2$, and $v_{xx}^2$ only if the function $v(x,t)$ is
bounded away from zero. As was observed in Remark~\ref{rem:lowerbd}, this is the
case in particular if $v_0(x) \ge \delta > 0$ for some $\delta > 0$. However, we
are also interested in initial data which do not have that property. In
particular we may want to consider the situation where $(u_0,v_0) = (1,0)$ 
when $x < 0$ and $(u_0,v_0) = (0,1)$ when $x > 0$. In that case, the evolution
describes the diffusive mixing of the initially separated species $\cA$, $\cB$. 

To handle the case where the second component $v(x,t)$ is not bounded away
from zero, a possibility is to add to the energy density $e(x,t)$ a small 
multiple of $w^2$, where $w = 2u + v$. 

\begin{lem}\label{lem:edf34}
If $\theta > 0$ is sufficiently small, the quantities $e_1(x,t)$, $f_1(x,t)$, 
$d_1(x,t)$ defined by 
\begin{equation}\label{edf3}
\begin{split}
  e_1 \,&=\, e \,+\, \theta w^2/2\,, \\
  f_1 \,&=\, f \,+\, \theta b w w_x + 2\theta (a-b)w u_x\,, \\
  d_1 \,&=\, d \,+\, \theta b w_x^2 + 2\theta (a-b)w_x u_x\,,
\end{split}
\end{equation}
satisfy the energy balance $\partial_t e_1 = \partial_x f_1 - d_1$, and there 
exists a constant $c > 0$ such that 
\begin{equation}\label{cbd1}
  e_1 \,\ge\, c\bigl(u^2 + (1+v)v^2\bigr)\,, \qquad
  d_1 \,\ge\, c\bigl(u_x^2 + (1+v)v_x^2 + \rho^2\bigr)\,.
\end{equation}
Similarly the quantities $\tilde e_1(x,t)$, $\tilde f_1(x,t)$, $\tilde d_1(x,t)$ 
defined by 
\begin{equation}\label{edf4}
\begin{split}
  \tilde e_1 \,=\, \tilde e \,+\, \theta w_x^2/2\,, \quad
  \tilde f_1 \,=\, \tilde f \,+\, \theta w_x w_t\,, \quad
  \tilde d_1 \,=\, \tilde d \,+\, \theta b w_{xx}^2 + 2\theta (a{-}b)w_{xx} 
  u_{xx}\,, 
\end{split}
\end{equation}
satisfy the energy balance $\partial_t \tilde e_1 = \partial_x \tilde f_1 
- \tilde d_1$ as well as the lower bounds
\begin{equation}\label{cbd2}
  \tilde e_1 \,\ge\, c\bigl(u_x^2 + (1+v)v_x^2 + \rho^2\bigr)\,, \qquad
  \tilde d_1 \,\ge\, c\bigl(u_{xx}^2 + (1+v)v_{xx}^2 + \rho_x^2 + (1+v)\rho^2
  \bigr)\,.
\end{equation}
\end{lem}

\begin{proof}
Since $w_t = 2a u_{xx} + b v_{xx} = b w_{xx} + 2(a-b)u_{xx}$, it is straightforward 
to verify that the additional terms involving the parameter $\theta$ in
\eqref{edf3}, \eqref{edf4} do not destroy the local energy balances. The 
lower bounds \eqref{cbd1} are easily obtained using the definitions of 
the quantities $e_1$, $d_1$ and applying Young's inequality, provided   
$\theta > 0$ is sufficiently small (depending on $a,b$). Estimates
\eqref{cbd2} are obtained similarly, using in addition the lower bound
\eqref{dpos}. 
\end{proof}

\section{Uniformly local energy estimates}\label{sec4}

In this section we complete the proof of our main result,
Proposition~\ref{main2}, using the dissipative structures introduced in
Section~\ref{sec3}. We fix the parameters $a,b > 0$, and we consider a global
solution $(u,v) \in C^0([0,+\infty),X^2)$ of system \eqref{RD2} with initial
data $(u_0,v_0) \in X_+^2$, as given by Proposition~\ref{prop:exist}.  Following
our previous works \cite{GS1,GS2}, our strategy is to control the behavior of
the solution $(u(t),v(t))$ for large times using localized energy estimates.

For convenience, we first prove Proposition~\ref{main2} under the additional 
assumption that the solution of \eqref{RD2} satisfies
\begin{equation}\label{eq:lowv}
  \inf_{t \ge 0}\, \inf_{x \in \R}\,v(x,t) \,\ge\, \delta\,, \qquad\hbox{for some}
  ~\delta > 0\,.
\end{equation}
As was observed in Remark~\ref{rem:lowerbd}, this is the case if the initial 
function $v_0(x) = v(x,0)$ is bounded away from zero. Assumption~\ref{eq:lowv} 
allows us to use the relatively simple EDS structures \eqref{edf1},
\eqref{edf2bis} instead of the more complicated ones introduced in
Lemma~\ref{lem:edf34}, and this makes the argument somewhat easier to follow.
The proof is however completely similar in the general case, see
Section~\ref{ssec44} below.

Given $\epsilon > 0$ and $x_0 \in \R$, we define the localization function
\begin{equation}\label{chidef}
  \chi(x) \,=\, \chi_{\epsilon,x_0}(x) \,:=\, \frac{1}{\cosh\bigl(
  \epsilon(x-x_0)\bigr)}\,, \qquad x \in \R\,.
\end{equation}
This function is smooth and satisfies the bounds
\begin{equation}\label{chibd}
  0 \,<\, \chi(x) \le 1\,, \qquad |\chi'(x)| \,\le\, \epsilon \chi(x)\,,
  \qquad |\chi''(x)| \,\le\,  \epsilon^2 \chi(x)\,, \qquad x \in \R\,.
\end{equation}
Note also that $\int_\R \chi(x)\dd x = \pi/\epsilon$. The translation parameter
$x_0$ plays no role in the subsequent calculations, but at the end we shall take
the supremum over $x_0 \in \R$ to obtain uniformly local estimates.  In
contrast, the dilation parameter $\epsilon > 0$ is crucial, and will
be chosen in an appropriate time-dependent way.

\subsection{The localized energy and its dissipation}\label{ssec41}

We first exploit the EDS structure \eqref{edf1}. We fix some observation 
time $T > 0$ and we consider the localized energy 
\[
  E(t) \,=\, \int_\R \chi(x)\,e(x,t)\dd x \,=\, \int_\R \chi(x) 
  \Bigl(\frac12 u(x,t)^2 + \frac16 v(x,t)^3\Bigr)\dd x\,, \qquad
  t \in [0,T]\,.
\]
Note that this quantity is well defined thanks to the localization function 
$\chi$, which is integrable. If $t > 0$, we also introduce the associated 
energy dissipation
\[
  D(t) \,=\, \int_\R \chi(x)\,d(x,t)\dd x \,=\, \int_\R \chi(x) 
  \Bigl(a u_x(x,t)^2 + b v(x,t)v_x(x,t)^2 + \rho(x,t)^2\Bigr)\dd x\,.
\]
Since the solution $(u,v)$ of the parabolic system \eqref{RD2} is 
smooth for $t > 0$, it is straightforward to verify that $E \in 
C^0([0,T]) \cap C^1((0,T))$ and that
\begin{equation}\label{Eder}
\begin{split}
  E'(t) \,=\, \int_\R \chi(x)\,\partial_t e(x,t)\dd x \,&=\,
  \int_\R \chi(x) \Bigl(\partial_x f(x,t) - d(x,t)\Bigr)\dd x \\
  \,&=\, -\int_\R \chi'(x) f(x,t)\dd x - D(t)\,.
\end{split}
\end{equation}
To bound the flux term, we use \eqref{chibd} and the pointwise estimate 
$f^2 \le C_0\,e d$, where $C_0 > 0$ depends on the parameters $a,b$. 
Applying Young's inequality, we obtain
\[
  \bigg|\int_\R \chi'(x) f(x,t)\dd x\bigg| \,\le\, \epsilon \int_\R \chi 
  \bigl(C_0\,e d\bigr)^{1/2}\dd x \,\le\, \frac12\,D(t) + \frac{C_0 
  \epsilon^2}{2}\,E(t)\,.
\] 
At this point, we choose the dilation parameter $\epsilon$ so that 
\begin{equation}\label{epsdef}
  C_0\,\epsilon^2 \,=\, \frac{1}{T}\,.
\end{equation}
We thus obtain the differential inequality $E'(t) \le -\frac12 D(t) + 
\frac{1}{2T}E(t)$, which can be integrated using Gr\"onwall's lemma to 
give the useful estimate
\begin{equation}\label{Eineq}
  E(T) + \frac12 \int_0^T D(t)\dd t \,\le\, e^{1/2}\, E(0)\,.
\end{equation}

Next, we introduce integrated quantities related to the second EDS structure 
\eqref{edf2bis}. For all $t \in (0,T)$ we define
\begin{align*}
  \tilde E(t) \,&=\, \int_\R \chi(x)\,\tilde e(x,t)\dd x \,=\,
  \int_\R \chi(x) \Bigl(\frac{a{+}b}{2}\,u_x^2 + bv v_x^2 + \frac{1}{2}\,
  \rho^2\Bigr)(x,t)\dd x\,, \\
  \tilde D(t) \,&=\, \int_\R \chi(x)\,\tilde d_0(x,t)\dd x \,=\,  
  \int_\R \chi(x) \Bigl(\gamma u_{xx}^2 + \gamma v v_{xx}^2 
  + \gamma (1{+}4v)\rho^2 + b \rho_x^2\Bigr)(x,t) \dd x\,,
\end{align*}
where $\gamma > 0$ is as in Corollary~\ref{lem:dpos}. The same calculation as
in \eqref{Eder} leads to
\[
  \tilde E'(t) \,=\, -\int_\R \chi'(x) \tilde f(x,t)\dd x - \int_\R \chi(x)
  \,\tilde d(x,t)\dd x \,\le\, -\int_\R \chi'(x) \tilde f(x,t)\dd x - 
  \tilde D(t)\,,
\]
where the inequality follows from \eqref{dpos}. The difficulty here is that 
the flux term does not satisfy a pointwise estimate of the form $\tilde f^2
\le C\,\tilde e \tilde d_0$, see Remark~\ref{rem:flux}. However, we can 
decompose
\[
  \tilde f \,=\, \tilde f_0 - \frac{b^2}{3}\,v_x^3\,, \qquad 
  \hbox{where} \quad \tilde f_0 \,=\, (a{+}b)u_x u_t + 2b v v_x v_t 
  + b\rho \rho_x\,,
\]
and it is easy to check that $\tilde f_0^2\le C_1\,\tilde e \tilde d_0$ for some 
$C_1 > 0$. In particular, we find as before
\begin{equation}\label{flux1}
  \bigg|\int_\R \chi'(x) \tilde f_0(x,t)\dd x\bigg| \,\le\, \epsilon \int_\R 
  \chi \bigl(C_1\,\tilde e \tilde d_0\bigr)^{1/2}\dd x \,\le\, \frac14\,\tilde D(t) 
  + C_1 \epsilon^2\,\tilde E(t)\,.
\end{equation}
As for the term involving $v_x^3$, we integrate by parts to obtain the 
identity
\[
  \int_\R \chi' v_x^3 \dd x \,=\, -\int_\R \chi'' v v_x^2 \dd x - 
  2 \int_\R \chi' v v_x v_{xx} \dd x\,.
\]
Using \eqref{chibd} and Young's inequality, we deduce
\begin{equation}\label{flux2}
  \frac{b^2}{3}\,\bigg|\int_\R \chi'(x) v_x^3(x,t)\dd x\bigg| \,\le\, 
  \frac{b \epsilon^2}{3}\,\tilde E(t) +  \frac14\,\tilde D(t) + 
  \frac{4b^3 \epsilon^2}{9\gamma}\,\tilde E(t)\,.
\end{equation}
The combination of \eqref{flux1}, \eqref{flux2} gives the desired 
estimate on the flux term\:
\[
  \bigg|\int_\R \chi'(x) \tilde f(x,t)\dd x\bigg| \,\le\, \frac12\, 
  \tilde D(t) + C_2 \epsilon^2\,\tilde E(t)\,, \qquad \hbox{where}\quad 
  C_2 \,=\, C_1 + \frac{b}{3} + \frac{4b^3}{9\gamma}\,.
\]
Integrating the differential inequality $\tilde E'(t) \le -\frac12 \tilde 
D(t) + C_2 \epsilon^2 E(t)$ over the time interval $[t_0,T]$, where 
$t_0 \in [0,T]$, we arrive at the estimate
\begin{equation}\label{tEineq}
  \tilde E(T) + \frac12 \int_{t_0}^T \tilde D(t)\dd t \,\le\, C_3 \tilde 
  E(t_0)\,, \qquad t_0 \in [0,T]\,,
\end{equation}
where $C_3 = \exp(C_2 \epsilon^2 T) = \exp(C_2/C_0)$.  

Finally, we use the crucial fact that the EDS structures \eqref{edf1}, 
\eqref{edf2bis} are {\em ordered}, in the sense that 
\[
  \tilde e(x,t) \,\le\, C_4\,d(x,t)\,, \qquad \hbox{where}\qquad
  C_4 \,=\, \max\Bigl(1,\frac{a+b}{2a}\Bigr)\,.
\]
In particular, the inequality $\tilde E(t) \le C_4 D(t)$ holds for all 
$t \in (0,T)$. Thus, if we average \eqref{tEineq} over $t_0 \in [0,T]$ and use 
\eqref{Eineq}, we obtain
\begin{equation}\label{tEineq2}
  \tilde E(T) + \frac{1}{2T} \int_0^T t \tilde D(t)\dd t \,\le\, \frac{C_3}{T}\,
  \int_0^T \tilde E(t_0)\dd t_0 \,\le\, \frac{C_3 C_4}{T}\,\int_0^T D(t)\dd t 
  \,\le\, \frac{C_5}{T}\,E(0)\,,
\end{equation}
where the constant $C_5 = 2e^{1/2}\,C_3 C_4$ only depends on the parameters
$a,b$ in system~\eqref{RD2}, and is in particular independent of the observation
time $T > 0$ and of the solution $(u,v)$ under consideration. It is however
important to keep in mind that all integrated quantities $E, \tilde E$ and
$D, \tilde D$ depend implicitly on $T$ through the weight function \eqref{chidef}
and the choice \eqref{epsdef} of the parameter $\epsilon$. 

The bound \eqref{tEineq2} summarizes the information we can obtain from the EDS
structures \eqref{edf1}, \eqref{edf2bis}. It serves as a basis for all estimates
we shall derive on the solutions of \eqref{RD2} for large times. A typical 
application of \eqref{tEineq2} is\: 

\begin{lem}\label{lem:first}
There exist a constant $C_6 > 0$ depending only the parameters $a,b$ 
such that, for any solution $(u,v) \in C^0([0,+\infty),X^2)$ of 
\eqref{RD2} with initial data $(u_0,v_0) \in X_+^2$ and any $T > 0$, 
the following inequality holds\:
\begin{equation}\label{decay1}
  \sup_{x_0 \in \R} \,\int_{I(x_0,T)} \Bigl(u_x^2 + vv_x^2 + \rho^2\Bigr)
  (x,T)\dd x \,\le\, C_6\,R^3\,T^{-1/2}\,,
\end{equation}
where $I(x_0,T) = \bigl\{x \in \R\,;\,|x-x_0| \le (C_0 T)^{1/2}\bigr\}$ 
and $R = 1 + \|u_0\|_{L^\infty} + \|v_0\|_{L^\infty}$. 
\end{lem}

\begin{proof}
The initial energy density satisfies $e(x,0) \le \frac12 \|u_0\|_{L^\infty}^2 + 
\frac16 \|v_0\|_{L^\infty}^2 \le R^3$, so that
\begin{equation}\label{E0bd}
  E(0) \,=\, \int_\R \chi(x)\,e(x,0)\dd x \,\le\, R^3 \int_\R \chi(x)\dd x
  \,=\, \frac{\pi R^3}{\epsilon} \,=\, \pi R^3(C_0 T)^{1/2}\,.
\end{equation}
On the other hand, we have the lower bound $\tilde e \ge \gamma_1 (u_x^2 + 
v v_x^2 + \rho^2)$ for some constant $\gamma_1 > 0$, and it follows from
\eqref{chidef} that $\chi(x) \ge e^{-1}$ when $|x-x_0| \le \epsilon^{-1} = 
(C_0 T)^{1/2}$. We thus find
\[
  \tilde E(T) \,=\, \int_\R \chi(x)\,\tilde e(x,T)\dd x \,\ge\, 
  \gamma_1 e^{-1} \int_{I(x_0,T)} \Bigl(u_x^2 + vv_x^2 + \rho^2\Bigr)
  (x,T)\dd x\,.
\]
Applying \eqref{tEineq2} we deduce that
\[
  \int_{I(x_0,T)} \Bigl(u_x^2 + vv_x^2 + \rho^2\Bigr)
  (x,T)\dd x \,\le\, \frac{e\,C_5}{\gamma_1T}\, \pi R^3(C_0 T)^{1/2}\,,
\]
and taking the supremum over $x_0 \in \R$ in the left-hand side we 
arrive at \eqref{decay1}. 
\end{proof}

\begin{rem}\label{rem:decay1}
If $1 \le p < \infty$, the uniformly local space $L^p_\ul(\R)$ is defined 
as the set of all measurable functions $f : \R \to \R$ such that
\[
  \|f\|_{L^p_\ul} \,:=\, \biggl(\,\sup_{x_0 \in \R}\,\int_{|x-x_0|\le 1}|f(x)|^p \dd x
  \biggr)^{1/p} \,<\, \infty\,,
\]
see \cite{ABCD} for a nice review article on uniformly local spaces. In view 
of \eqref{eq:lowv}, the bound \eqref{decay1} implies that $\|u_x(t)\|_{L^2_\ul} + 
\|v_x(t)\|_{L^2_\ul} + \|\rho(t)\|_{L^2_\ul} \le CR^{3/2}t^{-1/4}$ for all $t \ge 1$. 
This estimate is far from optimal, but it already implies that the solution 
$(u,v)$ converges uniformly on compact sets to the family $\cE$ of spatially 
homogeneous equilibria, which is a nontrivial result. Using the smoothing
properties of the parabolic system \eqref{RD2}, it is possible to deduce 
analogous estimates in the uniform norm, in particular
\begin{equation}\label{uxvxbd}
  \|u_x(t)\|_{L^\infty} + \|v_x(t)\|_{L^\infty} \,\le\, \frac{C R^{7/4}}{t^{1/4}}\,,
  \qquad t \ge 2\,, 
\end{equation}
see Section~\ref{ssec43} below.  Note also that the optimal decay rates
for $u_x, v_x$ given by Proposition~\ref{main1} (in the particular case $a = b$)
indicate that the left-hand side of \eqref{decay1} indeed decays like $T^{-1/2}$
as $T \to +\infty$, so that \eqref{decay1} is not far from optimal. 
\end{rem}

\subsection{Control of the second order derivatives}\label{ssec42}

So far we only used the first term $\tilde E(T)$ in the left-hand side of
inequality \eqref{tEineq2}, but the integral term involving $\tilde D(t)$ is
also valuable. In particular, the bounds \eqref{tEineq2}, \eqref{E0bd} together
imply that $\tilde D(t) \le CR^3T^{-3/2}$ for ``most'' times $t$ in the interval
$[0,T]$, but that information is difficult to exploit because the exceptional
times where such a bound possibly fails may depend on the translation
parameter $x_0 \in \R$. This difficulty is inherent to our approach, and to avoid
it we extract from \eqref{tEineq2} a somewhat weaker estimate which is valid for
all times.

To do that, we first study the linear parabolic system
\begin{equation}\label{UVsys}
  U_t \,=\, a U_{xx} + 2v V - U\,, \qquad
  V_t \,=\, b V_{xx} + 2U - 4v V\,,
\end{equation}
which is obtained by differentiating \eqref{RD2} (where $k = 1$) with respect 
to the space coordinate $x$ or the time variable $t$. In the analysis of
\eqref{UVsys}, we consider the nonnegative function $v(x,t)$ as given,
independently of the solution $(U,V)$. The property we need is\:

\begin{lem}\label{lem:deriv}
There exists a constant $C_7 > 0$ depending only on the parameters $a,b$
such that, for any $v \in C^0([0,T],X_+)$ and any initial data $(U_1,V_1) \in X^2$ 
at time $t_1 \in [0,T]$, the solution $(U,V) \in C^0([t_1,T],X^2)$ of 
\eqref{UVsys} satisfies
\begin{equation}\label{derivbd}
  \int_\R \chi(x) \Bigl(2 |U(x,T)| + |V(x,T)|\Bigr)\dd x \,\le\, 
  C_7 \int_\R \chi(x) \Bigl(2 |U_1(x)| + |V_1(x)|\Bigr)\dd x\,,
\end{equation}
where $\chi$ is given by \eqref{chidef} with $\epsilon > 0$ as in 
\eqref{epsdef}. 
\end{lem}

\begin{proof}
Since the function $v(x,t)$ is nonnegative, the linear system \eqref{UVsys} is 
cooperative, so that a (component-wise) comparison principle holds as for 
the original system \eqref{RD}. In particular, the solution $(U,V)$ satisfies
the estimates $|U(x,t)| \le \tilde U(x,t)$ and $|V(x,t)| \le \tilde V(x,t)$, 
where $(\tilde U,\tilde V)$ denotes the solution of \eqref{UVsys} with initial 
data $(|U_1|,|V_1|)$ at time $t_1$. In other words, it is sufficient to prove 
\eqref{derivbd} for nonnegative initial data $(U_1,V_1)$, in which case the 
solution $(U,V)$ remains nonnegative by the maximum principle. 

Fix $t_1 \in [0,T]$, $(U_1,V_1) \in X_+^2$, and let $(U,V) \in C^0([t_1,T],X^2)$
be the solution of \eqref{UVsys} such that $(U(t_1),V(t_1)) = (U_1,V_1)$. 
Integrating by parts and using \eqref{chibd}, we easily find
\begin{align*}
  \frac{\D}{\D t}\int_\R \chi \bigl(2U + V)\dd x \,&=\, \int_\R \chi 
  \bigl(2a U_{xx} + b V_{xx}\bigr)\dd x \\
  \,&=\, \int_\R \chi'' \bigl(2a U + b V\bigr)\dd x \,\le\, 
  \epsilon^2 c \int_\R \chi\bigl(2U + V\bigr)\dd x\,,
\end{align*}
where $c = \max(a,b)$. This differential inequality is then integrated on the time 
interval $[t_1,T]$ to give \eqref{derivbd} with $C_7 = \exp(c\epsilon^2 T) = 
\exp(c/C_0)$. 
\end{proof}

Returning to the nonlinear system~\eqref{RD2}, we apply Lemma~\ref{lem:deriv} 
to estimate first the time derivatives $u_t, v_t$, and then the quantity 
$\rho = u - v^2$. 

\begin{lem}\label{lem:uvrho}
Under the assumption~\eqref{eq:lowv}, any solution $(u,v) \in C^0([0,+\infty),X^2)$ of 
\eqref{RD2} with initial data $(u_0,v_0) \in X_+^2$ satisfies, for any $T > 0$, 
\begin{equation}\label{decay4}
  \sup_{x_0 \in \R} \,\int_{I(x_0,T)} \Bigl(|u_{xx}| + |v_{xx}|  + |\rho|\Bigr)
  (x,T)\dd x \,\le\, C_8\,R^3\,T^{-1/2}\,,
\end{equation}
where $I(x_0,T)$ and $R$ are as in Lemma~\ref{lem:first}, and the constant 
$C_8 > 0$ depends only on $a,b,\delta$. 
\end{lem}

\begin{proof}
We start from inequality \eqref{tEineq2}, and we choose a time $t_1 \in [T/2,T]$ 
where the continuous function $t \mapsto t\tilde D(t)$ reaches its minimum 
over the interval $[T/2,T]$. We then have
\begin{equation}\label{tD}
  \frac{T}{8}\,\tilde D(t_1) \,\le\, \frac{1}{2T} \int_0^T t \tilde D(t)\dd t
 \,\le\, \frac{C_5}{T}\, E(0)\,, \qquad \hbox{hence} \quad 
 \tilde D(t_1) \,\le\, \frac{8C_5}{T^2}\, E(0)\,.
\end{equation}
We recall that $\tilde D = \int_\R \chi \tilde d_0 \dd x$, where $\tilde d_0$ 
is defined in \eqref{dpos}. Under assumption \eqref{eq:lowv}, there exists 
a constant $\gamma_2 > 0$ (depending only on $a,b,\delta$) such that
$\tilde d_0 \ge \gamma_2 (u_t^2 + v_t^2)$. Therefore, using H\"older's inequality 
and estimate \eqref{tD}, we find 
\begin{align*}
  \int_\R \chi \bigl(2|u_t(t_1)| + |v_t(t_1)|\bigr)\dd x \,&\le\, 
  \biggl(\int_\R \chi\dd x\biggr)^{1/2} \biggl(5\int_\R \chi \bigl(u_t(t_1)^2 
  + v_t(t_1)^2\bigr)\dd x\biggr)^{1/2} \\
  \,&\le\, \Bigl(\frac{5\pi}{\gamma_2 \epsilon}\Bigr)^{1/2}\,
  \tilde D(t_1)^{1/2}  \,\le\, \Bigl(\frac{5\pi}{\gamma_2 \epsilon}\Bigr)^{1/2}
  \Bigl(\frac{8C_5}{T^2}\Bigr)^{1/2}\, E(0)^{1/2}\,.
\end{align*}
Note that the right-hand side is of the form $C T^{-3/4} E(0)^{1/2}$, where the 
constant depends only on $a,b,\delta$. We now apply Lemma~\ref{lem:deriv}
to $(U,V) = (u_t,v_t)$, and we deduce from \eqref{E0bd}, \eqref{derivbd} that
\begin{equation}\label{decay2}
  \int_\R \chi(x) \bigl(2|u_t(x,T)| + |v_t(x,T)|\bigr)\dd x \,\le\, 
  C\,T^{-3/4}\,E(0)^{1/2} \,\le\, C_9\,R^{3/2}\,T^{-1/2}\,,
\end{equation}
where the constant $C_9 > 0$ only depends on $a,b,\delta$. Note that 
estimate \eqref{decay2} holds at the observation time $T$, and not 
at the intermediate time $t_1$ on which we have poor control. 

To complete the proof of \eqref{decay4}, it remains to control the quantity 
$\rho = u - v^2$, which measures the distance to the chemical equilibrium. It 
is straightforward to verify that $\rho$ satisfies the evolution equation
\begin{equation}\label{rhoeq}
  \rho_t \,=\, a \rho_{xx} - \bigl(1 + 4rv\bigr)\rho 
  + 2 \bigl(r{-}1\bigr)v v_t + 2 a v_x^2\,,
\end{equation}
where $r = a/b$. Since $v(x,t) \ge 0$, the maximum principle implies that 
$|\rho(x,t)| \le \bar \rho(x,t)$ for all $x \in \R$ and all $t \in 
[0,T]$, where $\bar \rho$ is the solution of simplified equation
\[
  \bar \rho_t \,=\, a \bar \rho_{xx} - \bar \rho + 2 |r{-}1| |v v_t| + 2 a v_x^2\,,
\]
with initial data $\bar \rho(x,0) = |\rho(x,0)|$. If we denote $c = 
2\max(|r{-}1|,a)$, we thus have
\begin{align*}
  \frac{\D}{\D t}\int_\R \chi \bar \rho\dd x \,&=\, a \int_\R \chi'' \bar \rho\dd x
  - \int_\R \chi \bar \rho\dd x + c \int_\R \chi \bigl(|v v_t| + v_x^2\bigr)\dd x \\
  \,&\le\, \bigl(a\epsilon^2 - 1\bigr)\int_\R \chi \bar \rho\dd x + c \int_\R \chi 
  \bigl(|v v_t| + v_x^2\bigr)\dd x\,. 
\end{align*}
Integrating that inequality over the time interval $[0,T]$ and using 
\eqref{epsdef}, \eqref{decay1}, \eqref{decay2}, we obtain
\begin{equation}\label{decay3}
\begin{split}
  \int_\R \chi|\rho(T)|\dd x \,&\le\, C\,e^{-T} \int_\R \chi|\rho(0)|\dd x
  + C \int_0^T e^{-(T-t)} \int_\R \chi\Bigl(|v(t)| |v_t(t)| + v_x(t)^2\Bigr)
  \dd x\dd t \\
 \,&\le\, C \,e^{-T} R^2 T^{1/2} + C \int_0^T e^{-(T-t)} \Bigl(R^{5/2}\,t^{-1/2} + 
  R^3\,t^{-1/2}\Bigr)\dd t \\ \,&\le\, C_{10}\,R^3\,T^{-1/2}\,, 
\end{split}
\end{equation}
where the constant $C_{10}$ only depends on $a,b,\delta$. 

Finally, since $au_{xx} = u_t + \rho$ and $b v_{xx} = v_t - 2\rho$, estimate 
\eqref{decay4} follows immediately from \eqref{decay2}, \eqref{decay3} after 
taking the supremum over $x_0 \in \R$. 
\end{proof}

\begin{rem}\label{rem:decay2} 
Estimate \eqref{decay2} implies that $\|u_{xx}(t)\|_{L^1_\ul} + \|v_{xx}(t)\|_{L^1_\ul}
+ \|\rho(t)\|_{L^1_\ul} \le CR^3 t^{-1/2}$ for $t \ge 1$, and using parabolic 
smoothing one deduces that
\begin{equation}\label{uvrho}
  \|u_{xx}(t)\|_{L^\infty} + \|v_{xx}(t)\|_{L^\infty} + \|\rho(t)\|_{L^\infty} 
  \,\le\, \frac{C R^{7/2}}{t^{1/2}}\,, \qquad t \ge 2\,,
\end{equation}
see Section~\ref{ssec43}. However, in view of Remark~\ref{rem:higher}, 
we believe that these decay rates are suboptimal. Note that the optimal rates 
conjectured in \eqref{higherest} suggest that the left-hand side of 
\eqref{decay4} indeed decays like $T^{-1/2}$ as $T \to +\infty$, 
so that \eqref{decay4} is not far from optimal.
\end{rem}

\subsection{From uniformly local to uniform estimates}\label{ssec43}

Lemmas~\ref{lem:first} and \ref{lem:uvrho} give apparently optimal
estimates on the quantities $u_x, v_x$ in some (time-dependent) uniformly 
local $L^2$ norm, and on $u_{xx}, v_{xx}, \rho$ in some uniformly local $L^1$
norm. To conclude the proof of Proposition~\ref{main2}, it remains to convert
these estimates into ordinary $L^\infty$ bounds, as already announced in
Remarks~\ref{rem:decay1} and \ref{rem:decay2}. The starting point is the
following well-known estimate for the heat semigroup
$S(t) = \exp(t\partial_x^2)$ acting on uniformly local spaces. If
$f \in L^p_\ul(\R)$ for some $p \in [1,+\infty)$, then
\begin{equation}\label{heatul}
  \|S(t)f\|_{L^\infty(\R)} \,\le\, C \min\bigl(1,t^{-1/(2p)}\bigr) 
  \|f\|_{L^p_\ul(\R)}\,, \qquad t > 0\,,
\end{equation}
see \cite[Proposition~2.1]{ABCD}. In particular, for short times, we have 
exactly the same parabolic smoothing effect for the solutions of the 
heat equation as in the ordinary $L^p$ spaces. It is easy to establish a 
similar result for the solutions of the linearized system \eqref{UVsys}. 

\begin{lem}\label{UVsmoothing}
Assume that $(U,V)$ is a solution of \eqref{UVsys}, where $\|v(t)\|_{L^\infty} 
\le R$ for some $R \ge 1$. Given $p \in [1,\infty)$, there exists a constant 
$C_{11} \ge 1$ depending only on $a,b,p$ such that, for all $t_1 > t_0 \ge 0$ 
satisfying $C_{11} R (t_1 - t_0) \le 1$, the following estimate holds\:
\begin{equation}\label{UVestLp}
  \|U(t)\|_{L^\infty} + \|V(t)\|_{L^\infty} \,\le\, \frac{C_{11}}{(t-t_0)^{1/(2p)}}
  \Bigl(\|U(t_0)\|_{L^p_\ul} + \|V(t_0)\|_{L^p_\ul}\Bigr)\,, \qquad
  t_0 < t \le t_1\,.
\end{equation}
\end{lem}

\begin{proof}
Without loss of generality, we can take $t_0 = 0$. We denote $W = (U,V)$ 
and we assume that the initial data $W_0 = (U_0,V_0)$ belong to $L^p_\ul(\R)^2$ for 
some $p \in [1,+\infty)$. If we write equation \eqref{UVsys} in integral 
form and use estimate \eqref{heatul}, we easily obtain
\[
  \|W(t)\|_{L^\infty} \,\le\, \frac{C}{t^{1/(2p)}}\,\|W_0\|_{L^p_\ul} + 
  C R \int_0^t \|W(s)\|_{L^\infty}\dd s\,, \qquad t > 0\,.
\]
Setting $\|W\| = \sup\bigl\{t^{1/(2p)}\|W(t)\|_{L^\infty}\,;\, 0 < t \le t_1\bigr\}$, 
we find $\|W\| \le C \|W_0\|_{L^p_\ul} + C' R t_1 \|W\|$, for some positive 
constants $C,C'$. If we now choose $t_1 > 0$ so that $C'R t_1 \le 1/2$, we 
conclude that $\|W\| \le 2C \|W_0\|_{L^p_\ul}$, which is the desired estimate. 
\end{proof}

We first apply Lemma~\ref{UVsmoothing} to $(U,V) = (u_x,v_x)$, with $p = 2$. 
As was observed in Remark~\ref{rem:decay1}, we know from \eqref{decay1} that 
$\|u_x(t)\|_{L^2_\ul} + \|v_x(t)\|_{L^2_\ul} + \|\rho(t)\|_{L^2_\ul} \le CR^{3/2}t^{-1/4}$ 
for all $t \ge 1$. Thus, taking $t \ge 2$ and choosing $t_0 = t-1/(C_{11}R)
\ge t/2$, we see that \eqref{UVestLp} implies estimate \eqref{uxvxbd}.  
Similarly, we can apply Lemma~\ref{UVsmoothing} to $(U,V) = (u_t,v_t)$, with 
$p = 1$. Here we invoke estimate \eqref{decay2}, which implies that 
$\|u_t(t)\|_{L^1_\ul} + \|v_t(t)\|_{L^1_\ul} \le C R^{3/2}t^{-1/2}$ for all $t \ge 1$, 
and choosing $t, t_0$ as above we deduce from \eqref{UVestLp} that
\begin{equation}\label{utvtbd}
  \|u_t(t)\|_{L^\infty} + \|v_t(t)\|_{L^\infty} \,\le\, \frac{C R^2}{t^{1/2}}\,, 
  \qquad t \ge 2\,.
\end{equation}

To control the quantity $\rho = u - v^2$ in $L^\infty(\R)$, we can proceed
as in the proof of Lemma~\ref{lem:uvrho}. Integrating \eqref{rhoeq} 
on the time interval $[2,t]$ and using estimates \eqref{uxvxbd}, \eqref{utvtbd}, 
we easily obtain
\begin{equation}\label{rhobd}
\begin{split}
  \|\rho(t)\|_{L^\infty} \,&\le\, e^{-(t-2)} \|\rho(2)\|_{L^\infty} + 
  c \int_2^t e^{-(t-s)}\,\Bigl(\|v(s)\|_{L^\infty}\|v_t(s)\|_{L^\infty}
  + \|v_x(s)\|_{L^\infty}^2\Bigr)\dd s \\
  \,&\le\, R^2\,e^{-(t-2)} + C \int_2^t e^{-(t-s)}\,\Bigl(\frac{R^3}{s^{1/2}}
  + \frac{R^{7/2}}{s^{1/2}}\Bigr)\dd s \,\le\, \frac{C R^{7/2}}{t^{1/2}}\,,
  \quad t \ge 2\,.
\end{split}
\end{equation} 
As $au_{xx} = u_t + \rho$ and $b v_{xx} = v_t - 2\rho$, we obtain \eqref{uvrho} 
from estimates \eqref{utvtbd}, \eqref{rhobd}. In addition, since $|\rho(x,t)| 
\le R^2$ for all $t \ge 0$ by \eqref{uvbound}, we deduce from \eqref{rhobd}
that $\|\rho(t)\|_{L^\infty} \le C R^{7/2}(1+t)^{-1/2}$ for all $t \ge 0$, 
which is the second estimate in \eqref{mainest2}. 

The only remaining step consists in improving the decay rates of the first-order 
derivatives $u_x, v_x$, so as to obtain the first estimate in \eqref{mainest2}. 

\begin{lem}\label{lem:uvbetter}
Under the assumption~\eqref{eq:lowv}, any solution $(u,v) \in C^0([0,+\infty),X^2)$ of 
\eqref{RD2} with initial data $(u_0,v_0) \in X_+^2$ satisfies, for any $T > 0$, 
\begin{equation}\label{uvbetter}
  \|u_x(t)\|_{L^\infty} + \|v_x(t)\|_{L^\infty} \,\le\, \frac{C_{12} R^{7/2}}{t^{1/2}}
  \,\log(2+t)\,, \quad t > 0\,,
\end{equation}
where $R = 1 + \|u_0\|_{L^\infty} + \|v_0\|_{L^\infty}$ and the constant $C_{12}$ 
depends only on $a,b,\delta$. 
\end{lem}

\begin{proof}
Since $u_t = a u_{xx} - \rho$, we have the integral representation
\begin{equation}\label{uxint}
  u_x(t) \,=\, \partial_x S(at/2) u(t/2) - \int_{t/2}^t \partial_x S(a(t{-}s)) 
  \rho(s)\dd s\,, \qquad t > 0\,,
\end{equation}
where $S(t) = \exp(t\partial_x^2)$ is the heat semigroup. The first term in 
the right-hand side is easily estimated\:
\begin{equation}\label{uxsimple}
  \|\partial_x S(at/2) u(t/2)\|_{L^\infty} \,\le\, \frac{C}{t^{1/2}}\,
  \|u(t/2)\|_{L^\infty} \,\le\, \frac{CR}{t^{1/2}}\,.
\end{equation}
To bound the integral term in \eqref{uxint}, we distinguish two cases, 
according to whether $s \ge t-1$ or $s < t-1$ (if $t \le 2$, the second
possibility is excluded). 

\medskip\noindent
{\bf Case 1\:} $s \ge \max(t-1,t/2)$. Since $\|\partial_x S(t)f\|_{L^\infty} \le 
C t^{-1/2}\|f\|_{L^\infty}$, we obtain using \eqref{rhobd}
\begin{equation}\label{case1est}
  \bigl\|\partial_x S(a(t{-}s))\rho(s)\bigr\|_{L^\infty} \,\le\, 
  \frac{C}{(t-s)^{1/2}}\,\|\rho(s)\|_{L^\infty} \,\le\, \frac{C}{(t-s)^{1/2}}\,
  \frac{R^{7/2}}{(1+s)^{1/2}}\,.
\end{equation}

\medskip\noindent
{\bf Case 2\:} $t \ge 2$ and $t/2 \le s \le t-1$. Here we observe that, for all 
$x \in \R$, 
\begin{align*}
  \bigl|\partial_x S(a(t{-}s))\rho(t)\bigr|(x) \,&\le\, 
   \frac{C}{(t-s)^{1/2}}\int_\R \exp\biggl(-\frac{|x-y|^2}{4a(t{-}s)}\biggr)
   \frac{|x-y|}{t-s}\,|\rho(y,s)|\dd y \\ 
   \,&\le\, 
   \frac{C}{t-s}\int_\R \exp\biggl(-\frac{|x-y|^2}{5a(t{-}s)}\biggr)
   \,|\rho(y,s)|\dd y \\ 
    \,&\le\, \frac{C}{t-s}\int_\R \frac{|\rho(y,s)|}{\cosh\bigl(\epsilon(s) 
   |x-y|\bigr)}\dd y\,,\quad \hbox{where}\quad \epsilon(s) \,=\, 
   \frac{1}{(C_0 s)^{1/2}}\,. 
\end{align*}
In the last line, we used the assumption that $t-s \le s$ and the fact 
that, for any $\gamma > 0$, there exists $C > 0$ such that $e^{-x^2}
\le C\cosh(\gamma x)^{-1}$ for all $x \in \R$. Now, we know from 
\eqref{decay3} that
\[
  \sup_{x \in \R}\,\int_\R \frac{|\rho(y,s)|}{\cosh\bigl(\epsilon(s)|x-y|\bigr)}\dd y
  \,\le\, C_{10} R^3 s^{-1/2} \,\le\, \frac{C R^3}{(1+s)^{1/2}}\,, 
\]
and we conclude that  
\begin{equation}\label{case2est}
  \bigl\|\partial_x S(a(t{-}s))\rho(s)\bigr\|_{L^\infty} \,\le\, 
  \frac{C}{t-s}\,\frac{R^3}{(1+s)^{1/2}}\,.
\end{equation}

Combining \eqref{case1est}, \eqref{case2est} we can estimate the 
integral term in \eqref{uxint} as follows\:
\begin{align*}
  \int_{t/2}^t \bigl\|\partial_x S(a(t{-}s))\rho(s)\bigr\|_{L^\infty} \dd s
  \,&\le\, \frac{CR^{7/2}}{(1+t)^{1/2}}\int_{t/2}^t \min\biggl(\frac{1}{t-s}\,,\,
  \frac{1}{(t-s)^{1/2}}\biggr)\dd s \\
  \,&\le\, \frac{C R^{7/2}}{(1+t)^{1/2}}\log(2+t)\,,
\end{align*}
and using in addition \eqref{uxsimple} we obtain the desired estimate for 
$\|u_x(t)\|_{L^\infty}$. The bound on $\|v_x(t)\|_{L^\infty}$ is obtained by 
a similar argument.
\end{proof}

\subsection{The case where $v$ is not bounded away from zero}\label{ssec44}

We briefly indicate here how the arguments of Sections~\ref{ssec41}--\ref{ssec43}
have to be adapted to establish Proposition~\ref{main2} without assuming that the 
second component $v(x,t)$ of system~\eqref{RD2} is bounded away from zero. 
As already mentioned, the idea is to use the modified EDS structures introduced
in Remark~\ref{lem:edf34}, where the additional parameter $\theta > 0$ is chosen 
sufficiently small, depending on $a,b$. It is straightforward to verify that the flux term 
$f_1(x,t)$ in \eqref{edf3} still satisfies the bound $f_1^2 \le C_0 e_1 d_1$ for some 
$C_0 > 0$, so that the proof of inequality \eqref{Eineq} is unchanged. Similarly, 
the additional flux term $\theta w_x w_t$ in \eqref{edf4} is harmless, because
\[
  \bigl(w_x w_t\bigr)^2 \,=\, w_x^2 \bigl(bw_{xx} + 2(a-b)u_{xx}\bigr)^2
  \,\le\, C\, \tilde e_1 \tilde d_1\,,
\]
for some constant $C > 0$. As a consequence, the proof of the crucial inequality
\eqref{tEineq2} is not modified either. In view of the improved lower bounds 
\eqref{cbd2}, the conclusion of Lemma~\ref{lem:first} is strengthened as follows\:
\[
  \sup_{x_0 \in \R} \,\int_{I(x_0,T)} \Bigl(u_x^2 + (1+v)v_x^2 + \rho^2\Bigr)
  (x,T)\dd x \,\le\, C_6\,R^3\,T^{-1/2}\,,
\]
for some constant $C_6 > 0$ depending only on $a,b,\theta$. Similarly, 
Lemma~\ref{lem:uvrho} holds without assuming \eqref{eq:lowv} and with
the stronger conclusion
\[
  \sup_{x_0 \in \R} \,\int_{I(x_0,T)} \Bigl(|u_{xx}| + (1+v)|v_{xx}|  + |\rho|\Bigr)
  (x,T)\dd x \,\le\, C_8\,R^3\,T^{-1/2}\,,
\]
where the constant $C_8$ only depends on $a,b,\theta$. The rest of the proof 
of Proposition~\ref{main2} does not rely on assumption~\ref{eq:lowv}, and 
follows exactly the same lines as in Section~\ref{ssec43}. 

\section{Stability analysis of spatially homogeneous equilibria}
\label{sec5}

In this section we study the solutions of system~\eqref{RD} in a neighborhood
of a spatially homogeneous equilibrium $(\bar u,\bar v)$ with $\bar u = \bar v^2$ 
and $\bar v > 0$. We look for solutions in the form
\[
  u(x,t) \,=\, \bar u\bigl(1 + 4 \tilde u(x,t)\bigr)\,, \qquad
  v(x,t) \,=\, \bar v\bigl(1 + 2 \tilde v(x,t)\bigr)\,,
\]
so that the perturbations $\tilde u, \tilde v$ satisfy the system
\begin{equation}\label{RD3}
\begin{split}
  \tilde u_t(x,t) \,&=\, a \tilde u_{xx}(x,t) + k_1\bigl(\tilde v(x,t)
  - \tilde u(x,t) + \tilde v(x,t)^2\bigr)\,, \\
  \tilde v_t(x,t) \,&=\, b \tilde v_{xx}(x,t) + k_2\bigl(\tilde u(x,t)
  - \tilde v(x,t) - \tilde v(x,t)^2\bigr)\,,
\end{split}
\end{equation}
where $k_1 = k$ and $k_2 = 4k\bar v$. We introduce the matrix notation
\[
  W \,=\, \begin{pmatrix} W_1 \\ W_2\end{pmatrix} \,\equiv\,
 \begin{pmatrix} \tilde u \\ \tilde v\end{pmatrix}\,,\qquad
  \cN \,=\, \begin{pmatrix} -1 \\ 1\end{pmatrix}\,, \qquad
  \cM \,=\, \begin{pmatrix} k_1 \\ -k_2\end{pmatrix}\,, \qquad
  D \,=\, \begin{pmatrix} a & 0 \\ 0 & b\end{pmatrix}\,,
\]
so that \eqref{RD3} takes the simpler form
\begin{equation}\label{RD4}
  W_t \,=\, D W_{xx} + \bigl(\cN \cdot W + W_2^2\bigr)\cM\,.
\end{equation}
Note that the reaction terms in \eqref{RD4} are always proportional
to the vector $\cM$, which therefore spans the ``stoichiometric subspace''
of the chemical reaction. They vanish when $\cN\cdot W + W_2^2 = 0$,
so that the tangent space to the manifold $\cE$ of equilibria at
the origin $W = 0$ is orthogonal to the vector $\cN$. 

The integral equation associated with \eqref{RD4} is
\begin{equation}\label{Duhamel}
  W(t) \,=\, \cS(t) * W_0 + \int_0^t \cS(t-s)\cM * W_2(s)^2 \dd s\,,
  \qquad t > 0\,,
\end{equation}
where $*$ denotes the convolution with respect to the space variable
$x \in \R$, and $\cS(t) = \cS(\cdot,t)$ is the matrix-valued function
defined by
\begin{equation}\label{Sdef}
  \cS(x,t) \,=\, \frac{1}{2\pi}\int_\R \exp\bigl(t A(\xi)\bigr)
  \,e^{i\xi x}\dd x\,, \qquad x \in \R\,,\quad t > 0\,,
\end{equation}
with
\begin{equation}\label{Adef}
  A(\xi) \,=\, -D\xi^2 + \cM \,\cN^\top \,=\, \begin{pmatrix} -k_1 -a\xi^2
  & k_1 \\ k_2 & -k_2 -b\xi^2\end{pmatrix}\,.
\end{equation}
The exponential in \eqref{Sdef} can be computed explicitly. For that
purpose, it is convenient to introduce the notation
\[
  \mu \,=\, \frac{a+b}{2}\,, \qquad \nu \,=\, \frac{a-b}{2}\,, \qquad
  \kappa \,=\, \frac{k_1+k_2}{2}\,, \qquad \ell \,=\, \frac{k_1-k_2}{2}\,,
\]
so that $a = \mu + \nu$, $b = \mu - \nu$, $k_1 = \kappa+\ell$, $k_2 = 
\kappa - \ell$. We observe that
\[
  A(\xi) \,=\, - (\kappa+\mu\xi^2)\1 + B(\xi)\,, \qquad \hbox{where}\qquad
  B(\xi) \,=\, \begin{pmatrix} -\ell-\nu\xi^2
  & k_1 \\ k_2 & \ell + \nu \xi^2\end{pmatrix}\,.
\]
Moreover $B(\xi)^2 = \Delta(\xi)^2\1$, where $\1$ is the identity matrix and
\begin{equation}\label{Deltadef}
  \Delta(\xi) \,=\, \sqrt{\kappa^2 + 2 \ell \nu \xi^2 + \nu^2 \xi^4}
  \,=\, \sqrt{k_1 k_2  + (\ell + \nu \xi^2)^2}\,.
\end{equation}
In particular, the eigenvalues of $A(\xi)$ are real and equal to $\lambda_\pm(\xi)
= -(\kappa+\mu \xi^2) \pm \Delta(\xi)$. Using these observations, it is easy to
verify that
\begin{equation}\label{expA}
  \exp\bigl(t A(\xi)\bigr) \,=\, e^{-(\kappa+\mu\xi^2)t}\Bigl(\cosh\bigl(\Delta(\xi)
  t\bigr) \,\1 \,+\, \frac{\sinh(\Delta\bigl(\xi) t\bigr)}{\Delta(\xi)}
  \,B(\xi)\Bigr)\,, \qquad t \ge 0\,.  
\end{equation}

The following result specifies the decay rate of the kernel $\cS(\cdot,t)$ in 
$L^1(\R)$ as $t \to +\infty$. 

\begin{prop}\label{prop:Sest}
For any integer $m \in \N$, there exists a constant $C > 0$ such that the 
matrix-valued function $S(\cdot,t)$ defined by \eqref{Sdef} satisfies, 
for all $t > 0$, the estimates 
\begin{equation}\label{Sest}
\begin{split}
  \|\partial_x^m \cS(t)\|_{L^1(\R)} \,&\le\, C t^{-m/2}\,, \\
  \|\partial_x^m \cS(t)\cM\|_{L^1(\R)} \,&\le\, C t^{-m/2}\bigl(
  e^{-2\kappa t} + |\nu| t^{-1}\bigr)\,, \\
  \|\partial_x^m \cN^\top\cS(t)\|_{L^1(\R)} \,&\le\, C t^{-m/2}\bigl(
  e^{-2\kappa t} + |\nu| t^{-1}\bigr)\,, \\
  \|\partial_x^m \cN^\top\cS(t)\cM\|_{L^1(\R)} \,&\le\, C t^{-m/2}\bigl(
  e^{-2\kappa t} + \nu^2 t^{-2}\bigr)\,.
\end{split}
\end{equation}
\end{prop}

\begin{proof}
The following interpolation estimate will be repeatedly used\: if $f : \R 
\to \C$ is integrable and if the Fourier transform $\hat f$ belongs to the 
Sobolev space $H^1(\R)$, then
\begin{equation}\label{finterp}
  \|f\|_{L^1}^2 \,\le\, C\|f\|_{L^2} \|x f\|_{L^2} \,\le\, C  
  \|\hat f\|_{L^2} \|\partial_\xi \hat f\|_{L^2}\,. 
\end{equation}
Of course, inequality \eqref{finterp} remains valid if $f$ is vector-valued or 
matrix-valued. Given any $t > 0$, we first apply \eqref{finterp} to $f(x) = S(x,t)$, 
recalling that $\hat f(\xi) = \hat \cS(\xi,t) = \exp(tA(\xi))$ is given 
by \eqref{expA}. Without loss of generality, we assume henceforth that $a \ge b$, 
so that $\nu \ge 0$ (the converse case is completely similar). Using the 
elementary bounds
\[
  \max\bigl(\sqrt{k_1k_2}\,,\,|\ell + \nu\xi^2|\bigr) \,\le\, \Delta(\xi)
  \,\le\, \kappa + \nu \xi^2\,,
\]
as well as $\cosh(z) \le e^z$ and $\sinh(z) \le \min(1,z)e^z$ for $z \ge 0$, 
we easily deduce from \eqref{expA} the pointwise estimates
\[
  |\hat \cS(\xi,t)| \,\le\, C\,e^{-b\xi^2 t}\,, \qquad
  |\partial_\xi\hat \cS(\xi,t)| \,\le\, C|\xi| t\,e^{-b\xi^2 t}\,,
\]
which imply that $\|\hat\cS(t)\|_{L^2} \le C t^{-1/4}$ and $\|\partial_\xi \hat\cS(t)\|_{L^2} 
\le C t^{1/4}$. It thus follows from \eqref{finterp} that the $L^1$ norm of 
$\cS(t)$ is uniformly bounded for all $t > 0$, and repeating the same argument 
with $f(x) = \partial_x^m S(x,t)$ for some $m \in \N$ we arrive at the first 
inequality in \eqref{Sest}. 

The other inequalities in \eqref{Sest} exploit cancellations that 
occur when the matrix $S(x,t)$ acts on the vector $\cM$ (to the right) 
or on the vector $\cN^\top$ (to the left). We start from the identities
\begin{equation}\label{MNid}
  B(\xi)\cM \,=\, -\kappa \cM - \nu \xi^2 \begin{pmatrix} k_1 \\ k_2
  \end{pmatrix}\,, \qquad
  \cN^\top B(\xi)\,=\, -\kappa \cN^\top + \nu \xi^2 \begin{pmatrix} 1 & 1
  \end{pmatrix}\,,
\end{equation}
which follow immediately from the definitions. Writing $\cosh(\Delta t) = 
e^{-\Delta t} + \sinh(\Delta t)$ in \eqref{expA}, we find
\begin{equation}\label{SMprelim}
  \hat \cS(\xi,t)\cM \,=\, e^{-(\kappa+\mu\xi^2)t}\left\{e^{-\Delta t}\cM + 
  \Bigl(1 - \frac{k}{\Delta}\Bigr)\sinh(\Delta t)\cM - \nu\xi^2\,
  \frac{\sinh(\Delta t)}{\Delta}\begin{pmatrix} k_1 \\ k_2
  \end{pmatrix}\right\}\,.
\end{equation}
In the particular case where $\nu = 0$, one has $\Delta = \kappa$, so that 
$\hat \cS(\xi,t)\cM = e^{-(2\kappa +\mu \xi^2)t}\cM$. In general, only the 
first term in the right-hand side of \eqref{SMprelim} decays exponentially 
in time, and can be estimated using the elementary bound $\mu \xi^2 + 
\Delta(\xi) \ge \kappa + b\xi^2$. The remaining terms are treated as above, 
and we arrive at pointwise estimates of the form
\begin{align*}
  |\hat \cS(\xi,t)\cM| \,&\le\, e^{-(2\kappa +b\xi^2)t} + C\nu\xi^2\,e^{-b\xi^2 t}\,,\\
  |\partial_\xi\hat \cS(\xi,t)\cM| \,&\le\, C|\xi|t\,e^{-(2\kappa +b\xi^2)t}
  + C\nu|\xi|(1+\xi^2t)\,e^{-b\xi^2 t}\,. 
\end{align*}
Invoking \eqref{finterp}, we thus obtain the second inequality in \eqref{Sest}. 
The third one is obtained similarly, starting from the second relation in 
\eqref{MNid}. 

Finally, a straightforward calculation shows that
\[
  \cN^\top \hat \cS(\xi,t)\cM \,=\, -2\,e^{-(\kappa+\mu\xi^2)t}\left\{
  \kappa\,e^{-\Delta t} + \Bigl(\kappa - \frac{\kappa^2 + \ell\nu \xi^2}{\Delta}\Bigr)
  \sinh(\Delta t)\right\}\,,
\]
and we deduce the pointwise estimates
\begin{align*}
  |\cN^\top \hat \cS(\xi,t)\cM| \,&\le\, C\,e^{-(2\kappa + b\xi^2)t} + C\nu^2\xi^4
  \,e^{-b\xi^2 t}\,,\\
  |\partial_\xi\cN^\top \hat \cS(\xi,t)\cM| \,&\le\, C|\xi|t\,e^{-(2\kappa 
  +b \xi^2)t} + C\nu^2|\xi|^3(1+\xi^2t)\,e^{-b\xi^2 t}\,. 
\end{align*}
Using again \eqref{finterp}, we obtain the last inequality in \eqref{Sest}. 
\end{proof}

The conclusion of Proposition~\ref{prop:Sest} is interesting for at least two 
reasons. First, if $a \neq b$ and if $W(t) = S(t)*W_0$ is a solution of the 
{\em linearized} equation \eqref{RD4} with initial data $W_0 \in X^2$, the first 
inequality in \eqref{Sest} (with $m = 1$) and the third one (with $m = 0$) 
imply that
\begin{equation}\label{linest}
  \|\tilde u_x(t)\|_{L^\infty} + \|\tilde v_x(t)\|_{L^\infty} \,=\, \cO(t^{-1/2})\,, 
  \qquad \|\tilde u(t) - \tilde v(t)\|_{L^\infty} \,=\, \cO(t^{-1})\,, 
\end{equation}
as $t \to +\infty$. We emphasize that, at the linear level, the difference 
$\tilde u - \tilde v$  measures the distance to the manifold $\cE$ of equilibria. 
Because of \eqref{linest}, we conjecture that the decay rates in \eqref{mainest1}
are optimal for general solutions of \eqref{RD}, see the discussion after
Proposition~\ref{main2}. Note that Proposition~\ref{main1} assumes that 
the diffusivities are equal, in which case Proposition~\ref{prop:Sest} 
shows that the difference $\tilde u(t) - \tilde v(t)$ decays exponentially
fast as $t \to +\infty$ when $W = (\tilde u,\tilde v)$ solves the linearized
equation. 

The second observation concerns the full, nonlinear equation \eqref{RD4}.  Using
the first two estimates in \eqref{Sest}, it is easy to prove by a fixed point
argument that the Cauchy problem for \eqref{RD4} is globally well-posed for
small data $W_0 \in L^p(\R)^2$ if $p < \infty$, and that the solutions satisfy
$\|W(t)\|_{L^\infty} = \cO(t^{-1/(2p)})$ as $t \to +\infty$. However, the
critical case $p = \infty$, which is relevant in the context of the present
paper, cannot be treated by this approach. In fact, using the optimal decay
estimates listed in Proposition~\ref{prop:Sest}, we are not even able to show
that the solution $W(t)$ of \eqref{RD4} originating from small initial data
$W_0 \in X^2$ stays uniformly bounded for all times, except in the case of equal
diffusivities where the problem is much simpler. The reason is that, if
$a \neq b$, the quantity $\|\cS(t)\cM\|_{L^1(\R)}$ decays like $t^{-1}$ as
$t \to +\infty$ and is therefore not integrable in time. This indicates that
the dynamics of system~\eqref{RD} in the space of bounded functions on $\R$ 
is not simple to analyze, even in a neighborhood of a spatially homogeneous 
equilibrium. 

\section{Conclusion and perspectives}\label{sec6}

The present work is only a modest incursion into the realm of extended
reaction-diffusion systems with a local gradient structure. Even for the very
simple example \eqref{RD}, which has many specific properties, our results are
incomplete and a global understanding of the dynamics is still missing.  To
be more precise, assume that the decay rates \eqref{higherest} hold for all
bounded and nonnegative solutions of \eqref{RD}, which is a reasonable conjecture
(although we are not able to prove that when $a \neq b$).  The quantity
$\rho = u - v^2$, which measures the distance to the manifold $\cE$ of
equilibria, satisfies the equation
\begin{equation}\label{rhoequation}
  \rho_t \,=\, a \rho_{xx} - k(1{+}4v)\rho + 2(a{-}b)v v_{xx} + 2a v_x^2\,.
\end{equation}
According to \eqref{higherest}, the last three terms in \eqref{rhoequation}
decay like $t^{-1}$ when $t \to +\infty$, whereas $\rho_{xx} = \cO(t^{-2})$. 
It is therefore reasonable to expect that
\begin{equation}\label{slaving}
  \rho \,=\, \frac{1}{k(1{+}4v)}\Bigl(2(a{-}b)v v_{xx} + 2a v_x^2\Bigr)
  + \cO(t^{-2})\,, \qquad t \to +\infty\,.
\end{equation}
Inserting this ansatz into the $v$-equation $v_t = bv_{xx} + 2k\rho$ and 
neglecting the higher-order terms, we obtain the following quasilinear 
diffusion equation
\begin{equation}\label{diffeq1}
  v_t \,=\, \frac{b+4av}{1+4v}\,v_{xx} + \frac{4a}{1+4v}\,v_x^2\,,
  \qquad x \in \R\,, \quad t > 0\,.
\end{equation}
Alternatively, setting $w = v + 2v^2$, we can write \eqref{diffeq1} in 
the more elegant form 
\begin{equation}\label{diffeq2}
  w_t \,=\, \bigl(D(w)w_x\bigr)_x\,, \qquad \hbox{where}\qquad 
  D(w) \,=\, a + \frac{b-a}{\sqrt{1+8w}}\,.
\end{equation}
We conjecture that the long-time asymptotics of any solution of \eqref{RD} in
$X_+^2$ corresponds to a slow motion along the manifold $\cE$ of chemical
equilibria, which is described to leading order by the diffusion equation
\eqref{diffeq1} or \eqref{diffeq2}. Note that the effective diffusion $D(w)$ in
\eqref{diffeq2} depends on the solution $w$ in a nontrivial way, except in the
particular case $a = b$ where \eqref{diffeq2} reduces to the linear heat
equation. Given two positive constants $w_\pm$, one can solve the Cauchy problem
for \eqref{diffeq2} with Riemann-like initial data 
\[
  w_0(x) \,=\, \begin{cases} w_- & \hbox{if } x < 0\,, \\
  w_+ & \hbox{if } x > 0\,,\end{cases}
\]
and this produces a self-similar solution of \eqref{diffeq2} which should 
describe the {\em diffusive mixing} of two chemical equilibria under the 
dynamics of \eqref{RD}, see \cite{GM} for a similar result in the context
of the Ginzburg-Landau equation. A rigorous justification of the slaving 
ansatz \eqref{slaving} and of the relevance of the diffusion equation 
\eqref{diffeq2} for the long-time asymptotics of the original system 
\eqref{RD} is left to a future work. 

On the other hand, the model we consider is just a simple example in a broad
class of systems, and it is natural to ask to which extent our analysis 
relies on specific features of \eqref{RD}. In a first step towards greater
generality, we consider the reaction $n\cA \xrightleftharpoons[]{} m\cB$, 
where $n,m$ are positive integers such that $n + m \ge 3$. The corresponding system
\begin{equation}\label{RDnm}
  u_t \,=\, a u_{xx} + nk\bigl(v^m-u^n\bigr)\,, \qquad
  v_t \,=\, b v_{xx} + mk\bigl(u^n-v^m\bigr)\,,
\end{equation}
is still cooperative, and the analogue of Proposition~\ref{prop:exist} holds.
It is also possible to find a polynomial EDS structure of the form \eqref{edf1}, 
which reads
\begin{align*}
  e \,&=\, \frac{1}{n(n{+}1)}\,u^{n+1} + \frac{1}{m(m{+}1)}\,v^{m+1}\,, \\
  f \,&=\, \frac{a}{n} u^n u_x + \frac{b}{m}\,v^m v_x\,, \\[2mm]
  d \,&=\, a u^{n-1} u_x^2 + b v^{m-1} v_x^2 + (u^n-v^m)^2\,.
\end{align*}
However, there is apparently less flexibility for constructing a second EDS 
structure in the sense of Section~\ref{sec3}, and at the moment we can do
that only if the ratio $a/b$ is not too different from $1$. Except for that 
limitation in the choice of the parameters $a,b$, the analogue of 
Proposition~\ref{main2} holds with a similar proof. 

The situation changes significantly when we turn our attention to more realistic 
chemical reactions such as $\cA_1 \xrightleftharpoons[]{} \cA_2 + \cA_3$. 
The associated system is still relatively simple
\begin{equation}\label{RD3x3}
  u_t \,=\, a u_{xx} - u + vw\,, \qquad
  v_t \,=\, b v_{xx} + u - vw\,, \qquad
  w_t \,=\, c w_{xx} + u - vw\,,
\end{equation}
but new difficulties arise that make the analysis substantially more
difficult. First, system~\eqref{RD3x3} is not cooperative, and does not satisfy
any comparison principle we know of. As a consequence, new arguments are needed
to show that the solutions of \eqref{RD3x3} stay uniformly bounded for all
nonnegative initial data in $L^\infty(\R)$. For the same reason, it is not
obvious that a solution starting close (in the $L^\infty$ sense) to a chemical
equilibrium will stay in a neighborhood of that equilibrium for all times. Next,
the only EDS structure we are aware of is given by the general formulas
\eqref{edfgen}, and we are not able to construct a second EDS structure that
controls the entropy dissipation, as we did in Section~\ref{sec3} for the
simpler system \eqref{RD}. At the moment, we are thus unable to prove the
analogue of Proposition~\ref{main2} for system~\eqref{RD3x3}, and a fortiori for
more complex reaction-diffusion systems of the form \eqref{RDgen}. We hope to be
able to elucidate some of these questions in the future.

\vfill\eject

\bigskip\noindent
{\bf Thierry Gallay}\\ 
Institut Fourier, Universit\'e Grenoble Alpes, 100 rue des Maths, 38610 Gi\`eres, 
France\\
Email\: {\tt Thierry.Gallay@univ-grenoble-alpes.fr}

\bigskip\noindent
{\bf Sini\v{s}a Slijep\v{c}evi\'c}\\
Department of Mathematics, University of Zagreb, Bijeni\v{c}ka 30, 
10000 Zagreb, Croatia\\
Email\: {\tt sinisa.slijepcevic@math.hr}

\end{document}